\documentclass{amsart}
\usepackage{color}
\usepackage{amsmath}
\usepackage{amssymb}
\usepackage{amsthm}
\usepackage{amscd}
\usepackage{enumitem}
\usepackage{tikz}
\usepackage{tikz-cd}
\usepackage{hyperref}
\usepackage[utf8]{inputenc}
\usepackage[all]{xy}

\usepackage{hyperref}
\hypersetup{
  colorlinks   = true,          %Colors links instead of ugly boxes
  urlcolor     = blue,          %Color for external hyperlinks
  linkcolor    = blue,          %Color of internal links
  citecolor   = red             %Color of citations
}

\theoremstyle{plain}
\newtheorem{theorem}{Theorem}[section]

\newtheorem{proposition}[theorem]{Proposition}
\newtheorem{corollary}[theorem]{Corollary}

\newtheorem{rem}[theorem]{Remark}
\newtheorem{Cor}[theorem]{Corollary}

\newtheorem{Lem}[theorem]{Lemma}
\newtheorem{rems}[theorem]{Remarks}

\theoremstyle{definition}
\newtheorem{definition}[theorem]{Definition}

\numberwithin{equation}{section}

\newcommand\fantome[1]{}

\def \N {\mathbb N}
\def \Z {\mathbb Z}

\def \C {\mathbb C}

\def \F {\mathbb F}

\DeclareMathOperator{\ev}{ev}

\DeclareMathOperator{\Hom}{Hom}
\DeclareMathOperator{\Fitt}{Fitt}

\DeclareMathOperator{\Tr}{Tr}

\DeclareMathOperator{\Frac}{Frac}
\DeclareMathOperator{\Gal}{Gal}
\DeclareMathOperator{\Ker}{Ker}
\DeclareMathOperator{\Lie}{Lie}
\DeclareMathOperator{\Vol}{Vol}
\DeclareMathOperator{\Spec}{Spec}
\DeclareMathOperator{\MSpec}{MSpec}
\DeclareMathOperator{\id}{id}

\makeatletter
\@namedef{subjclassname@2020}{%
  \textup{2020} Mathematics Subject Classification}
\makeatother
\author{Tiphaine Beaumont}

\address{
Universit\'e de Caen Normandie,
Laboratoire de Math\'ematiques Nicolas Oresme,
CNRS UMR 6139,
Campus II, Boulevard Mar\'echal Juin,
B.P. 5186,
14032 Caen Cedex, France.
}
\email{tiphaine.beaumont@unicaen.fr}

\title{On equivariant class formulas for $t$-modules}

\begin{document}

\begin{abstract}
We obtain an equivariant class formula for $z$-deformation of $t$-modules. Under mild conditions, it allows us to get an equivariant class formula for $t$-modules.
\end{abstract}



\date{\today}

\maketitle

\tableofcontents

%%%%%%%%%%%%%%%%%%%%%%%%

\section{Introduction}

In \cite{ref9}, Taelman introduced the notions of class module and unit module for Drinfeld modules and gave a conjectural class formula when $A=\F_q[\theta]$. He proved it later in  \cite{ref8}. 

 It was extended  by Fang in \cite{ref18} for Anderson modules and by Demeslay in \cite{ref16, ref17} for Anderson modules with variables. Mornev proved the class formula for some Drinfeld $A$-modules with general $A$  in \cite{ref21}.
 Recently, in \cite{ref1}, Anglès, Ngo Dac and Tavares Ribeiro proved the class formula for a general $A$ and some Anderson modules, in particular for Drinfeld modules. 

For an abelian Galois group $G$, the equivariant class formula was proved by differents ways when $p$ does not divide $\vert G\vert$ by Anglès and Taelman in \cite{ref10}, Anglès and Tavares Ribeiro in \cite{ref3} for Drinfeld modules. It was extended by Fang in \cite{ref20} when $p$ does not divide $\vert G\vert$ for Anderson modules.

Recently, Ferrara, Green, Higgins and D. Popescu in \cite{ref4} adapted the method of Taelman in \cite{ref8} to the equivariant theory for Drinfeld modules for general $G$.

This article is based on \cite{ref4}  with the utilisation of the $z$-deformation of $t$-modules \cite{ref19, ref2, ref3}. Evaluating at $z=1$, it enables us to get an equivariant class formula in some cases using the method of \cite{ref1} in the equivariant setting. 

The strategy of the proofs consist of Taelman's techniques used in the equivariant setting with variable. We combine the results and the appendix of Ferrara, Green, Higgins and D. Popescu (see \cite{ref4}) corresponding to the equivariant context with the results of Demeslay (see \cite{ref15}) which corresponds to the part with variables.

As many proofs follow the same line those in \cite{ref4} but for $\F_q(z)$ instead of $\F_q$ we give the statements and omit the proofs. 

\bigskip

Let us briefly describe the results of this paper. 

Let $p$ be a prime number and $q$ a power of $p$. Let $A=\F_q[\theta]$ with $\theta$ an indeterminate and 
 $k=\F_q(\theta)$ its field of fractions. We denote $k_\infty=\F_q\left(\left(\theta^{-1}\right)\right)$. Let $L$ be a finite extension of $k$.  We denote $L_\infty=L\otimes_k k_\infty$. 
We set the $\F_q$-algebra homomorphism $\tau:L_\infty\longrightarrow L_\infty$ which associates $x^q$ to $x$.

Let $K/k$ be a finite extension and $\mathcal{O}_K$ the integral closure of $A$ in $K$. Let $\mathcal{M}_n(K)\{\tau\}$ be the ring of twisted polynomials with coefficients in $\mathcal{M}_n(K)$. 
Let $E$ be a Anderson module of dimension $n$ defined over $\mathcal{O}_K$ : it means we take a $\F_q$-algebra homomorphism $\phi_E:	A \longrightarrow\mathcal{M}_n(\mathcal{O}_K)\{\tau\}$ which sends $\theta$ to $\sum\limits_{i=0}^rA_i\tau^i$
where for all $i\in [\![0;r]\!], A_i\in \mathcal{M}_n(\mathcal{O}_K)\{\tau\}$ and $A_0$ verifies $(A_0-\theta I_n)^n=0_n$. 
In particular, a Drinfeld module is an Anderson module of dimension 1.

Let $B$ be an $\mathcal{O}_K$-algebra. We denote by $E(B)$ the $A$-module $B^n$ equipped with the structure  of $A$-module induced by $\phi_E$. We also have the  $A$-module $B^n$ whose  structure of $A$-module is given by the morphism $\delta_E:	A \longrightarrow\mathcal{M}_n(\mathcal{O}_K)$ such that $\delta_E(\theta)=A_0$. We write it $\Lie_E(B).$

There exists a unique power serie $\exp_E\in I_n +\mathcal{M}_n(K)\{\tau\}\tau$ which verifies the equality $\exp_E\delta_E(\theta) =\phi_E(\theta) \exp_E$.  Moreover, it converges on $\Lie_E(L_\infty)$ if  $L/K$ is a finite extension.

\bigskip

We introduce the notion of almost taming module which generalizes the notion of taming module introduced by Ferrara, Green, Higgins and D. Popescu in \cite{ref4}.

\begin{definition} Let $L/K$ be a finite extension of abelian Galois group $G$. An {\it almost taming module} for $L/K$ is an $A$-module which verifies
\begin{itemize}
\item $M$ is an $A$-lattice of $L_\infty$,
\item $M$ is a projective  $A[G]$-module,
\item $M$ is an $\mathcal{O}_K\{\tau\}[G]$-module. 
\end{itemize}
\end{definition}

Following Taelman \cite{ref9}, we define $$U(E( M))=\{x\in \Lie_E(L_\infty), \exp_E(x)\in E(M)\}$$  as the unit module attached to $M$ and $$H(E( M))=\dfrac{E(L_\infty)}{ E(M)+\exp_E(\Lie_E(L_\infty))}$$ the class module for $M$. The unit module is an $A$-lattice of $\Lie_E(L_\infty).$
\bigskip

Anglès and Tavares Ribeiro introduced the notion of $z$-deformation for Drinfeld modules in \cite{ref3}. With Ngo Dac, they developped it and extended it for Anderson modules in \cite{ref19, ref2}. 
It allowed them to define the Stark units attached to $\mathcal{O}_L$ which we will extend to $M$.

Let $z$ be an indeterminate over $k_{\infty}$. We set $\widetilde{L}_\infty=L\otimes_{k}\widetilde{k}_\infty$ where $\widetilde{k}_\infty=\F_q(z)((\theta^{-1}))$.
We keep the notation $\tau$ for the $\F_q(z)$-algebra homomorphism $\tau:\widetilde{L}_\infty\longrightarrow\widetilde{L}_\infty$ which associates $x^q$ to $x$.

We recall that $E$ is an  Anderson module such that $\phi_E(\theta)=\sum\limits_{i=0}^r A_i\tau^i$. Then we can define $\widetilde{E}$ called the $z$-deformation of $E$ by the homomorphism of $\F_q(z)$-algebras $\phi_{\widetilde{E}}: \F_q(z)[\theta] \longrightarrow \mathcal{M}_n(K(z))\{\tau\}$ such that 
$\phi_{\widetilde{E}}(\theta)=\sum\limits_{i=0}^rz^i A_i\tau^i.$

If  $\exp_E=\sum\limits_{i\geq 0}E_i\tau^i$, we set $\exp_{\widetilde{E}}=\sum\limits_{i\geq 0}E_iz^i\tau^i$. It verifies $\exp_{\widetilde{E}}\delta_E(\theta) =\phi_{\widetilde{E}}(\theta) \exp_{\widetilde{E}}$. Furthermore, it converges on $\Lie_{\widetilde{E}}(\widetilde{L}_\infty)$. 
We denote $\widetilde{A}=\F_q(z)[\theta]$ and  $\widetilde{M}=M\otimes_A\widetilde{A}$.
Following Anglès, Ngo Dac and Tavares Ribeiro, we define   $$U(\widetilde{E}(\widetilde{M}))=\{x\in \Lie_{\widetilde{E}}(\widetilde{L}_\infty), \exp_{\widetilde{E}}(x)\in \widetilde{E}(\widetilde{M})\}.$$

By the same reasoning of \cite{ref16}, it is an $\widetilde{A}$-lattice of $\Lie_{\widetilde{E}}(\widetilde{L}_\infty)$.

Let  $\mathbb{T}_z(k_\infty)$ be the Tate algebra
 with coefficents in $k_\infty$ and $\mathbb{T}_z(L_\infty)=L_\infty\otimes_{k_\infty}\mathbb{T}_z(k_\infty). $ Note that  $\mathbb{T}_z(k_\infty)\subset \widetilde{L}_\infty.$

We define the module of $z$-units of $\widetilde{E}$ relative to $M[z]$ by  $$U(\widetilde{E}(M[z]))=\{x\in \Lie_{\widetilde{E}}(\mathbb{T}_z(L_\infty)), \exp_{\widetilde{E}}(x)\in \widetilde{E}(M[z])\}.$$

Following Anglès, Ngo Dac and Tavares Ribeiro in \cite{ref2} for Drinfeld modules and in \cite{ref19} for $t$-modules   we are now in position to define the module of Stark units of a $t$-module $E$ attached to  $M$.

\bigskip

We denote by $\ev:\mathbb{T}_z(L_\infty)\longrightarrow L_\infty$ the evaluation at $z=1$. The module of Stark units of $M$ is defined by $$U_{St}(E(M))=\ev(U(\widetilde{E}(M[z]))).$$ 
It is contained in   $U(E(M))$ and it is an $A$-lattice of $\Lie_E(L_\infty)$. 
\bigskip 

Following Ferrara, Green, Higgins and D. Popescu in \cite{ref4}, we define an equivariant regulator $[\Lambda: \Lambda']_G$   for two projective $A[G]$-modules $\Lambda, \Lambda'$ of $L_\infty^n$ or  for two projective $\widetilde{A}[G]$-modules $\Lambda, \Lambda'$ of $\widetilde{L}_\infty^n$.

We recall that $\MSpec(A)$ corresponds to the set of maximal ideals of $A$.

We define the $G$-equivariant $L$-function attached to $M$ by 
$$\mathcal{L}_G(E (M))=\prod\limits_{v\in \MSpec(A)}\dfrac{\vert \Lie_E(M/vM)\vert_G }{\vert E(M/vM)\vert_G} $$ where $\vert X \vert_G$ corresponds to the unique monic generator of $\Fitt_{A[G]}(X)$. In the same way, $\mathcal{L}_G(\widetilde{E} (\widetilde{M}))=\prod\limits_{v\in \MSpec(A)}\dfrac{\vert \Lie_{\widetilde{E}}(\widetilde{M}/v\widetilde{M}\vert_G }{\vert \widetilde{E}(\widetilde{M}/v\widetilde{M})\vert_G} $ where $\vert X \vert_G$ corresponds to the unique monic generator of $\Fitt_{\widetilde{A}[G]}(X)$.

Thanks to the introduction of $z$, the unit module becomes $G$-cohomologically trivial, which is not necessary the case without this variable.
Combining the method of Ferrara, Green, Higgins and D. Popescu in \cite{ref4} with the $z$-deformation, we get the following result.

\begin{theorem}
\textit{(Equivariant class number formula for $z$-deformation)}

We have
$$\left[\Lie_{\widetilde{E}}(\widetilde{M}):U(\widetilde{E}(\widetilde{M}))\right]_G=\mathcal{L}_G(\widetilde{E}( \widetilde{M})).$$
\end{theorem}

In particular, by following the method  in \cite{ref1}, we obtain the next theorem.

\begin{theorem} 
 Let $\Lambda$ be a projective $A[G]$-module such that $\Lambda\subset U_{St}(E(M))$ and $\Lambda$ is a $A$-lattice of $L_\infty$. 
 Let $E$ be a $t$-module defined over $\mathcal{O}_K$. Then
$$\mathcal{L}_G(E( M))^{-1}[ \Lie_E(M):\Lambda]_{A[G]}\in A[G].$$
Furthermore, 
$$\displaystyle\det_G\left(\dfrac{[ \Lie_E(M):\Lambda]_{A[G]}}{\mathcal{L}_G(E( M))}\right)=[ U_{St}(E(M)):\Lambda]_{A}.$$ 
\end{theorem}

With this theorem, we obtain the following corollary of which the last assertion could be deduced from \cite[Theorem 6.2.1 ]{ref4}.

\begin{corollary}
We denote $N=\Tr_G(M)$.  If $H(E(N))$ is trivial, then $U(E(M))$ and $U_{St}(E(M))$ are projective $A[G]$-modules.  We have

 $$[ \Lie_E(M):U_{St}(E(M))]_{A[G]}=\mathcal{L}_G(E( M)).$$
 
Furthermore, $$[ \Lie_E(M):U(E(M))]_{A[G]}\vert H(E(M))\vert_{A[G]}=\mathcal{L}_G(E( M)).$$
\end{corollary}

\textbf{Aknowledements : } The author thanks Bruno Anglès and Tuan Ngo Dac for the discussions that lead to this paper.

%%%%%%%%%%%%%%%%%%%%%%

\section{Background}
\subsection{Notation} ${}$\par

Let $p$ be a prime number and $q$ a power of $p$. Let $A=\F_q[\theta]$ with $\theta$ an indeterminate over $\F_q$ and 
 $k=\F_q(\theta)$ its field of fractions. We denote $k_\infty=\F_q\left(\left(\theta^{-1}\right)\right)$. We set $\C_\infty$ the completion of an algebraic closure of $k_\infty$. We set $v_\infty$ the valuation on $\C_\infty$ such that $v_\infty(\theta^{-1})=1$. 
 
Let  $z$ be an indeterminate over $\C_\infty$. 
We keep the notation  $v_\infty$ for the valuation on $\C_\infty(z)$  such that for $P\in\C_\infty[z]$ where $P(z)=\sum\limits_{i=0}^na_iz^i$ with  for all $i$, $a_i\in \C_\infty$ we have $v_\infty(P)=\min\{v_\infty(a_i), i\in [\![0;n]\!]\}$.

Let $K$ be a subfield of $\C_\infty$ such that $k_\infty\subset K$ and $K$ is complete with respect to  $v_\infty$. We denote by $\mathbb{T}_z(K)$ the completion of $K[z]$ for $v_\infty$, i.e.,  it corresponds to the elements of the form $\sum\limits_{i\geq0}a_iz^i$ where $a_i\in K$ and $\lim\limits_{i\to \infty} v_\infty(a_i)=\infty$. In particular, we have  $\mathbb{T}_z(k_\infty)=\F_q[z]\left(\left(\theta^{-1}\right)\right).$ 

We denote by $\widetilde{K}$ the completion of $K(z)$ for $v_\infty$. In particular, we have $\widetilde{k}_\infty=\F_q(z)\left(\left(\theta^{-1}\right)\right)$.  The $\F_q(z)$-vector space spanned by $\mathbb{T}_z(K)$ is dense in $\widetilde{K}$.

Let $K/k$ be a finite extension. We set $K_\infty=K\otimes_kk_\infty$ and $\mathcal{O}_K$ the integral closure of $A$ in $K$. Likewise, we set $\widetilde{K}_\infty=K\otimes_k\widetilde{k}_\infty$. We denote by $\widetilde{\mathcal{O}}_K$ the $\F_q(z)$-vector space spanned by $\mathcal{O}_K$ in $\widetilde{K}_\infty$ and $\widetilde{A}=\F_q(z)[\theta]$. In particular, $\widetilde{\mathcal{O}}_K$ is the integral closure of  $\widetilde{A}$ in $\widetilde{K}_\infty$. 

 We denote by $\tau:K_\infty\longrightarrow K_\infty$ the continuous morphism of $\F_q$-algebras which sends $x\in K_\infty$ to $x^q$.
 We still denote by $\tau:\widetilde{K}_\infty\longrightarrow \widetilde{K}_\infty$ the continuous morphism of $\F_q(z)$-algebras which sends $x\in \widetilde{K}_\infty$ to $x^q$.

For the rest of this paper, we take $\ell$ corresponding to $\F_q$ or $\F_q(z)$. 
We denote by $L_\ell=L\otimes_{k}\ell(\theta)$, $\ell_\infty=\ell\left(\left(\theta^{-1}\right)\right)$ and $L_{\infty,\ell}=L\otimes_k \ell_\infty$. We set $\mathcal{O}_{K,\ell}$ for $\mathcal{O}_{K}$ when $\ell=\F_q$ and $\widetilde{\mathcal{O}_{K}}$ when $\ell=\F_q(z)$. 

\subsection{Some projective modules} ${}$\par

In this section, we recall some definitions and results of Ferrara, Green, Higgins and D. Popescu in \cite[Section 7.2]{ref4}. We state them for $\F_q(z)$ and not just $\F_q$ as the arguments stay the same. We also generalize the notion of taming module of \cite{ref4} by an almost taming module. 

Let $L/k$ be a finite extension and $E$ a $t$-module defined over $\mathcal{O}_K$ where $K$ verifies $k\subset K$ and $L/K$ is a Galois extension of abelian group $G$. 

We recall that $l$ is $\F_q$ or $\F_q(z)$ and that $\N^*$ corresponds to the positive integers.

\begin{definition} 
Let $m\in \N^*$. 
\begin{itemize}
\item A {\it $\ell[\theta]$-lattice} in $L_{\infty,\ell}^m$ is a free $\ell[\theta]$-submodule $L_{\infty,\ell}^m$ of rank $m\dim_{\ell\left(\left(\theta^{-1}\right)\right)}L_{\infty,\ell}$ which generates $L_{\infty,\ell}^m$ as a $\ell\left(\left(\theta^{-1}\right)\right)$-vector space. 
\item A {\it $\ell[\theta][G]$-lattice} in $L_{\infty,\ell}^m$ is a $\ell[\theta][G]$-submodule of $L_{\infty,\ell}^m$ which is a $\ell[\theta]$-lattice of $L_{\infty,\ell}$.
\item A  {\it projective $\ell[\theta][G]$-lattice (respectively free)} in $L_{\infty,\ell}^m$ is a $\ell[\theta][G]$-lattice of $L_{\infty,\ell}^m$ which is  $\ell[\theta][G]$-projective (respectively free). 
 
\end{itemize}
\end{definition}

\begin{proposition}\cite[Prop. 7.2.1]{ref4}
We have the following assertions : 

\begin{itemize}

\item $L_{\infty,\ell}^m$ is a free $\ell_\infty[G]$-module.
\item If $\Lambda$ is a $\ell[\theta][G]$-lattice in $L_{\infty,\ell}^m$, $\ell(\theta)\Lambda$ is a free $\ell(\theta)[G]$-module.
\item For two $\ell[\theta][G]$-lattices $\Lambda_1, \Lambda_2$ of $L_{\infty,\ell}^m$ such that $\ell(\theta)\Lambda_1=\ell(\theta)\Lambda_2$, there exists a free $\ell[\theta][G]$-lattice $\Lambda$ of $L_{\infty,\ell}^m$ such that $\Lambda_1,\Lambda_2\subset \Lambda$. 
 
\end{itemize}
\end{proposition}

We generalize the definition of taming module introduced by  Ferrara, Green, Higgins and D. Popescu in \cite{ref4}  by the next definition. 

\begin{definition} An {\it almost taming module} for $L_\ell/K_\ell$ is a $\ell[\theta]$-module which verifies
\begin{itemize}
\item $M$ is a $\ell[\theta]$-lattice of $L_{\infty,\ell}$,
\item $M$ is a projective  $\ell[\theta][G]$-module,
\item $M$ is an $\mathcal{O}_{K,\ell}\{\tau\}[G]$-module. 
\end{itemize}
\end{definition}
In particular, a taming module is an almost taming module.

We denote by $\mathcal{O}_{K,\infty}$ (respectively $\mathcal{O}_{L,\infty}$) the intersection of valuation rings of infinite places of  $K$ (respectively $L$). 
 
 \begin{definition}\begin{enumerate}
\item \cite[Def. 7.2.7]{ref4}\label{modereinf} An  {\it$\infty$-taming module} for $L/K$ is a projective $\mathcal{O}_{K,\infty}[G]$-module of local constant rank 1 denoted   $\mathcal{W}^\infty$ such that $\mathcal{W}^\infty \subset \mathcal{O}_{L,\infty}$ and the quotient $\mathcal{O}_{L,\infty}/\mathcal{W}^\infty $ is finite.
 \item An {\it almost taming pair} is a couple $(M_\ell,\mathcal{W}^\infty$) where $M_\ell$ is an almost taming module for $L_\ell/K_\ell$ and $\mathcal{W}^\infty$ is a $\infty$-taming module.
 \end{enumerate}
\end{definition}

We can see that if $M$ is an almost taming module (respectively taming) for $L/K$ then $\widetilde{M}$ is an almost taming module (respectively taming) for $\widetilde{L}/\widetilde{K}$.

\subsection{Characteristic $p$ group-ring and cohomology }
In this section, we recall some statements of \cite[Section 7.1]{ref4} which will be useful to prove that a module is $\ell[\theta][G]$-projective by using group cohomology.

We set $G=H\times \Delta$ where $H$ is the $p$-Sylow of $G$. We assume that  $R$ is a Dedekind ring. For $\chi\in \Hom\left(\Delta,\overline{\Frac(R)}\right)$ where $\overline{\Frac(R)}$ is a separable closure of $\Frac(R)$, we denote by $\widehat{\chi}$ its equivalence class under $\chi\sim \sigma\circ\chi$ for $\sigma\in \Gal_{\overline{\Frac(R)}}$. We denote by $\widehat{\Delta}(R)$ all the equivalence classes.  We obtain the idempotents of $R[G]$,  indexed by these classes  $$e_{\widehat{\chi}}=\dfrac{1}{\vert\Delta\vert}\sum\limits_{\psi\in\widehat{\chi}, \delta\in\Delta}\psi(\delta)\delta^{-1}$$ for all $\widehat{\chi}\in \widehat{\Delta}(R)$. 

We have the ring isomorphism: $$R[G]=\bigoplus_{\widehat{\chi}\in \widehat{\Delta}(R)}e_{\widehat{\chi}}R[G]\cong \bigoplus_{\widehat{\chi}\in \widehat{\Delta}(R)}R(\chi)[H]$$ where $R(\chi)$ is the Dedekind ring obtained by adding the values of $\chi$.
 We use the isomorphism $e_{\widehat{\chi}}R[G]\cong R(\chi)[H]$ with the evaluation at $\chi$.
 For each $R[G]$-module $M$, in the same way, we obtain  $$M=\bigoplus_{\widehat{\chi}\in \widehat{\Delta}(R)}e_{\widehat{\chi}}M\cong \bigoplus_{\widehat{\chi}\in \widehat{\Delta}(R)}M^\chi$$ where $M^\chi=M\otimes_{R[G]}R(\chi)[H]$.

\begin{theorem} \cite[Corollary 7.1.7]{ref4} \label{Projct}
Let  $R$ be a Dedekind ring or a field of characteristic $p$, $G$ a finite abelian group and $M$ a finitely generated $R[G]$-module. 
\begin{itemize}
\item  $M$ is $R[G]$-projective if and only if $M$ is $R$-projective and $H$-cohomologically trivial if and only if $M$ is $R$-projective and $G$-cohomologically trivial.
\item If $R$ is a discrete valuation ring or a field then $R[G]$ is a semi-local ring of local direct summands $R(\chi)[H]$ for all $\widehat{\chi}$.

\item If $R$ is a discrete valuation ring or a field then $M$  is a free $R[G]$-module if and only if  $M$ is a projective $R[G]$-module of constant rank.

\end{itemize}

\end{theorem}
\subsection{Monic elements} ${}$\par
All this section was proven in \cite[Section 7.3]{ref4} when $\ell=\F_q$. The same arguments work for $\ell=\F_q(z)$. 
As $G$ is finite, we have $\ell[G]\left[\!\left[\theta^{-1}\right]\!\right]=\ell[\![\theta^{-1}]\!][G]$ and $\ell[G]\left(\left(\theta^{-1}\right)\right)=\ell\left(\left(\theta^{-1}\right)\right)[G]$.

\begin{definition} We define $\ell\left(\left(\theta^{-1}\right)\right)[H]^+$ as a sub-group of unitary elements $\ell\left(\left(\theta^{-1}\right)\right)[H]^\times$ by $$\ell\left(\left(\theta^{-1}\right)\right)[H]^+=\bigcup_{n\in \mathbb{Z}}\theta^n(1+\theta^{-1}\ell[H][\![\theta^{-1}]\!]).$$

\end{definition}
We use the decomposition of characters of $\ell[G]$ to obtain the direct sum $$\phi_\Delta:\ell[G]\left(\left(\theta^{-1}\right)\right)\cong \bigoplus_{\widehat{\chi}\in \widehat{\Delta}(l)}\ell(\chi)[H]\left(\left(\theta^{-1}\right)\right).$$
This allows us to have the next definition.

\begin{definition} We define $\ell\left(\left(\theta^{-1}\right)\right)[G]^+$ as a sub-group of the group of monic elements $\ell\left(\left(\theta^{-1}\right)\right)[G]^\times$ by $$\ell\left(\left(\theta^{-1}\right)\right)[G]^+=\phi_\Delta^{-1}\left(\bigoplus_{\widehat{\chi}\in  \widehat{\Delta}(l)}\ell(\chi)[H]\left(\left(\theta^{-1}\right)\right)^+\right).$$

\end{definition}

\begin{rem} We say that a polynomial $f\in \ell[G][\theta]$ is monic if $$f\in \ell[G][\theta]^+=\ell[G][\theta]\cap \ell\left(\left(\theta^{-1}\right)\right)[G]^+.$$ 
\end{rem}

\begin{theorem} (Weirstrass decomposition)
Let  $(\mathcal{O},m)$ be a complete local ring. Let $f\in\mathcal{O}[\![X]\!]\setminus m[\![X]\!]$ such that $f=\sum\limits_{i\in \mathbb{N}}a_iX^i$. Assume that $n$ is the smaller integer such that $a_n\not\in m$. Then $f$ has the unique Weierstrass decomposition $$f=(X^n+b_{n-1}X^{n-1}+\ldots+b_0).u \text{ with }b_i\in  m \text{ and }u\in\mathcal{O}[\![X]\!]^\times.$$

\end{theorem}

It permits us to obtain the next proposition. 
\begin{proposition} 
We have $$\ell\left(\left(\theta^{-1}\right)\right)[G]^\times=\ell\left(\left(\theta^{-1}\right)\right)[G]^+\times \ell[\theta][G]^\times.$$

\end{proposition}

\begin{Cor} \label{unitaire}
We have the isomorphism $$\ell\left(\left(\theta^{-1}\right)\right)[G]^\times\big/\ell[\theta][G]^\times=\ell\left(\left(\theta^{-1}\right)\right)[G]^+$$ which associates  to a class $g$ its unique monic representative $g^+$.

\end{Cor}

\subsection{Fitting ideals} ${}$\par

Let $R$ be a commutative ring and $M$ a finitely generated $R$-module.
Let $$R^a\longrightarrow R^b\longrightarrow M\longrightarrow 0$$ be a finite presentation of $M$ and $X$ be the matrix of $R^a\longrightarrow R^b$. Then we define  $\Fitt_R(M)$ as the ideal of $R$ spanned by the minors of size $b\times b$ if $b\leq a$ and $\Fitt_R(M)=0$ if $b\geq a$. $\Fitt_R(M)$ is independent of the choice of the finite presentation of $M$. 

\begin{itemize}
\item If $M\cong M_1\times M_2$ is the direct product of two $R$-modules of finite presentation then $\Fitt_R(M)=\Fitt_R(M_1)\Fitt_R(M_2)$.
\item If $R\longrightarrow R'$ is a ring homomorphism then $\Fitt_{R'}(R'\otimes_RM)=R'\otimes_R\Fitt_R(M)$.  
\item If $M_1\longrightarrow M\longrightarrow M_2\longrightarrow 0$ is exact  then $\Fitt_R(M_1)\Fitt_R(M_2)\subset \Fitt_R(M)$.

\end{itemize} 

Furthermore,  if $R$ is a Dedekind ring, then $\Fitt_R(M_1)\Fitt_R(M_2)= \Fitt_R(M)$. 

If $M$ is a  finitely generated and torsion $R$-module, there exist ideals $I_1,\ldots, I_n$  of $R$ such that $M\cong R/I_1\times \ldots
\times R/I_n$. We have $\Fitt_R(M)=\prod\limits_{i=1}^nI_i$. 
\bigskip

From now on,  $R$ is a noetherian semi-local ring. Let $N$  be a $R[\theta]$-module which is finitely generated and projective as a $R$-module. For example, as in \cite{ref4} $R=\mathbb{F}_q[G]$ and $N=\mathcal{M}/v$ for $\mathcal{M}$ a taming module for $L/K$ and $v\in \MSpec(A)$ or $N=\Lambda_1/\Lambda_2$ where $\Lambda_2\subseteq \Lambda_2$ are two projective $A[G]$-lattices of $L_\infty$. 

We can also take $R=\mathbb{F}_q(z)[G]$ or the same objects with $\mathcal{M}$ an almost taming module.

\begin{proposition}\cite[Prop. 7.4.1]{ref4} \label{fitt}
Let $N$ be a finitely generated $R[\theta]$-module and a projective $R$-module. 
\begin{itemize}
\item If $R$ is local and $rank_R(N)=m$ then $\Fitt_{R[\theta]}(N)$ is principal and has a unique monic generator denoted by $\vert N\vert_{R[\theta]}\in R[\theta]^+$. It has degree $m$ and is given by $\vert N\vert_{R[\theta]}=\det_{R[\theta]}(\theta I_m-A_\theta)$ where $A_\theta\in M_m(R)$ is the matrix of the $R$-endomorphism of $N$ given by the multiplication of $\theta$ in any $R$-basis of $N$. 

\item  If $R$ is semilocal (such that $R=\oplus_i R_i$) then  $\Fitt_{R[\theta]}(N)$ is principal and has a unique monic generator  $\vert N\vert_{R[\theta]}=\sum\limits_i\vert N\otimes_RR_i\vert_{R_i[\theta]}$ which belongs to $R[\theta]^+$. 

\end{itemize}

\end{proposition}

We recall that $\ell[G][\theta]^+=\ell[G][\theta]\cap \ell[G]\left(\left(\theta^{-1}\right)\right)^+.$
It allows us to have the following definition introduced in \cite{ref4} when $R=\F_q[G]$. 
\begin{definition} 
Let  $M$ be a $\ell[\theta][G]$-module which is a $l$-vector space of finite dimension and $G$-cohomologically trivial. Then we have $$\vert M\vert_G=\vert M\vert_{\ell[\theta][G]}\in \ell[G][\theta]^+.$$ 

\end{definition}

If there is the exact sequence of $\ell[\theta][G]$-modules
$$0\longrightarrow B\longrightarrow C \longrightarrow D\longrightarrow0$$

 with $B, C$ and $D$ which verify the conditions of the previous definition, then we have the equality $\vert C\vert_G=\vert B\vert_G\vert D\vert_G$. 

We will also use the notation $\vert X\vert_A$ for the monic generator of $\Fitt_A(X)$.  
\section{Anderson module and class formula}
In this section, we recall the definition of a $t$-module, the class module, the unit module and the Stark units. For this section, we take  $L/k$  a finite extension and $E$ a $t$-module defined over $\mathcal{O}_K$ where $K$ verifies $k\subset K\subset L$. 

Let $M$ be  an $A$-lattice of $L_\infty$ which is an $\mathcal{O}_K\{\tau\}$-module. It means that $M$ is stable over $E$. 
\subsection{Anderson module} ${}$\par

For $B\in \mathcal{M}_{n,m}(\C_{\infty})$ and $r\in \N$ such that $B=(B_{i,j})_{\substack{
i\in [\![1;n]\!] \\
j\in[\![1;m]\!]}}$, we denote $$\tau^r(B)=B^{(r)}=(\tau^r(B_{i,j}))_{\substack{
i\in [\![1;n]\!] \\
j\in[\![1;m]\!]}}.$$
Let $\mathcal{M}_n(\C_{\infty})\{\tau\}$ be the ring of twisted polynomials with coefficients in $\mathcal{M}_n(\C_{\infty})$ such that for $B,C\in \mathcal{M}_n(\C_{\infty})$ and $i,j\in \N$, we have $$B\tau^iC\tau^j=BC^{(i)}\tau^{i+j}.$$

For $X\in \C_{\infty}^n=\mathcal{M}_{n,1}(\C_{\infty})$ and $\sum A_i\tau^i\in\mathcal{M}_{n}(\C_{\infty})\{\tau\}$, we have $(\sum A_i\tau^i)X=\sum A_i\tau^i(X)$.

\begin{definition} 
Let $n,r\in \N^*$. A {\it $t$-module} $E$ of dimension $n$ consists of a $\F_q$-algebra homomorphism 

$$\begin{array}{ccccc}
\phi_E:	\F_q[\theta] &\to&\mathcal{M}_n(\C_\infty)\{\tau\} \\
\theta &\mapsto &\sum\limits_{i=0}^rA_i\tau^i
\end{array}$$
such that for all $i\in [\![0;r]\!], A_i\in \mathcal{M}_n(\C_{\infty})$ and $A_0$ verifies $(A_0-\theta I_n)^n=0_n$.

\end{definition}

A $t$-module is also called an Anderson $\F_q[t]$-module. We say that the Anderson module is defined on $\mathcal{O}_K$ where $k\subset K$ if for all $i\in [\![0;r]\!],  A_i\in \mathcal{M}_n(\mathcal{O}_K)$.

\begin{definition} 
A {\it Drinfeld module} is a $t$-module of dimension 1.

\end{definition}

We have the following lemma of Fang.

\begin{Lem}\cite[Lemma 1.4]{ref18} We have the following assertions : 
\begin{itemize}
\item $A_0^{q^n}=\theta^{q^n}I_n$, 
\item $\displaystyle\inf_{j\in\Z}(v_\infty(A_0^j)+j)$ is finite.
\end{itemize}
\end{Lem}

Let $B$ be an $\mathcal{O}_K$-algebra. We denote by $E(B)$ the $A$-module $B^n$ equipped with the structure of $A$-module induced by $\phi_E$ and $\Lie_E(B)$   the $A$-module $B^n$  whose  structure of $A$-module is given by the morphism 

$$\begin{array}{ccccc}
\delta_E:	A &\to&\mathcal{M}_n(\C_{\infty}) \\
\theta &\mapsto &A_0
\end{array}.$$
 
By the previous lemma, it can be extended to $$\begin{array}{ccccc}
\delta_E:	k_\infty &\to&\mathcal{M}_n(\C_{\infty}) \\
\sum_{i\geq m}a_i\theta^{-i} &\mapsto &\sum_{i\geq m}a_iA_0^{-i}
\end{array}.$$

We can endow $\Lie_E(L_\infty)$ with a structure of $k_\infty$-vector space.

There exists a unique power series $\exp_E\in I_n +\tau\mathcal{M}_n(K)\{\tau\}$ which verifies the equality $\exp_E\delta_E(\theta) =\phi_E(\theta) \exp_E$. Moreover, it converges on $\Lie_E(\C_\infty)$. 
\subsection{Modules and units} ${}$\par

Following Taelman \cite{ref8}, we define  
$$H(E( M))=\dfrac{E(L_\infty)}{ E(M)+\exp_E(\Lie_E(L_\infty))}$$ the {\it class module} for $M$ and $$U(E( M))=\{x\in \Lie_E(L_\infty), \exp_E(x)\in E(M)\}$$ the {\it unit module} attached to $M$.

We have the exact sequence of $A$-modules induced by $\exp_E$ : 

$$0\longrightarrow \dfrac{\Lie_E(L_\infty)}{U(E( M))}\longrightarrow \dfrac{E(L_\infty)}{E(M)}\longrightarrow  H(E( M))\longrightarrow 0.$$
\begin{proposition} We have the following assertions :  
\begin{itemize}
\item $H(E(M))$ is an $A$-module finitely generated and of torsion,
\item $U(E( M))$ is an $A$-lattice in $\Lie_E(L_\infty)$.
\end{itemize}
\end{proposition}
\begin{proof}
See \cite{ref9} for Drinfeld modules and \cite{ref18} for $t$-modules by replacing $\mathcal{O}_L$ by $M$. 

\end{proof}

Let $N, N'$ be two $A$-lattices of $L_\infty^n$. There exists $X\in Gl_m(k_\infty)$ such that $X\mathcal{B}=\mathcal{B'}$ where $\mathcal{B}$ and $\mathcal{B'}$ are bases of $N$ and $N'$ respectively. We define $[N:N']_A$ as the unique monic generator of the ideal $(\det(X))$. We can show that it does not depend on the choice of the bases. 
\\

 We look at a class formula  discovered by Taelman in \cite{ref8} for Drinfeld modules. Then it was proved by Fang in \cite{ref8} for $t$-modules and by Desmeslay in \cite{ref16} for $t$-modules with variables. 

 \begin{theorem}(The class formula)\label{classformula} The product 
$\mathcal{L}(E(M))=\prod\limits_{P\in \Spec(A)}\dfrac{\vert \Lie_E(M/PM)\vert_A}{\vert E(M/PM)\vert_A}$ converges in $k_\infty$. Furthermore, 
$$\mathcal{L}(E(M))=\left[\Lie_E(M):U(E(M))\right]_A\vert H(E(M))\vert_A.$$
\end{theorem}

\begin{proof}
It is the same as \cite{ref17} if $s=0$ and by replacing $\mathcal{O}_L$ by $M$.
\end{proof}
\subsection{Stark units} ${}$\par
In this section, we recall the $z$-deformation introduced by Anglès and Tavares Ribeiro in \cite{ref3} for Drinfeld modules and Anglès, Ngo Dac and Tavares Ribeiro in \cite{ref3} for $t$-modules and the definition of Stark units. 
\bigskip

Let $z$ be an indeterminate over $k_{\infty}$.

If  $E$ is a $t$-module such that $\phi_E(\theta)=\sum\limits_{i=0}^r A_i\tau^i$ then we can define a $t$-module $\widetilde{E}$ called the $z$-deformation of $E$ by the homomorphism of $\F_q(z)$-algebras :   

$$\begin{array}{ccccc}
\widetilde{\phi}	_E: \F_q(z)[\theta] &\to&\mathcal{M}_n(\widetilde{\C}_\infty)\{\tau\} \\
\theta &\mapsto &\sum\limits_{i=0}^rz^i A_i\tau^i
\end{array}.$$

Let $\exp_E$ be the unique element in $\mathcal{M}_n(K)\{\tau\}$ such that $\exp_E\equiv I_n \mod \tau$ and $\exp_E\delta_E(\theta) =\phi_E(\theta) \exp_E$. Likewise, if  $\exp_E=\sum\limits_{i\geq 0}E_i\tau^i$, we set $\exp_{\widetilde{E}}=\sum\limits_{i\geq 0}E_iz^i\tau^i$. We can show that it is the only element in $\mathcal{M}_n(K(z))\{\tau\}$ such that $\exp_{\widetilde{E}}\equiv I_n \mod \tau$ and $\exp_{\widetilde{E}}\delta_E(\theta) =\phi_{\widetilde{E}}(\theta) \exp_{\widetilde{E}}$. 

Furthermore, it converges on $\Lie_{\widetilde{E}}(\widetilde{L}_\infty)$. 
We denote by $\widetilde{M}=M\otimes_A\widetilde{A}$. We define   $U(\widetilde{E}(\widetilde{M}))=\{x\in \Lie_{\widetilde{E}}(\widetilde{L}_\infty), \exp_{\widetilde{E}}(x)\in \widetilde{E}(\widetilde{M})\}$  and  $H(\widetilde{E}(\widetilde{M}))=\dfrac{\widetilde{E}(\widetilde{L}_\infty)}{\widetilde{E}(\widetilde{M})+\exp_{\widetilde{E}}(\Lie_{\widetilde{E}}(\widetilde{L}_\infty))}$.

We have the following proposition proved by Desmeslay in \cite{ref16} when $M=\mathcal{O}_L$.  
  \begin{proposition} \label{reseautildea}\cite[Proposition 2.6]{ref16} We have the following assertions : 
 \begin{itemize}
 \item $H(\widetilde{E}(\widetilde{M}))$ is a $\F_q(z)$-vector space of finite dimension. 
 \item $U(\widetilde{E}(\widetilde{M}))$ is an $\widetilde{A}$-lattice of $\Lie_{\widetilde{E}}(\widetilde{L}_\infty)$. 
 \end{itemize}

\end{proposition}

We recall that $\mathbb{T}_z(k_\infty)$ is the Tate algebra
 with coefficents in $k_\infty$, meaning that $\mathbb{T}_z(k_\infty)=\F_q[z]\left(\left(\theta^{-1}\right)\right)$. We also have $\mathbb{T}_z(L_\infty)=L_\infty\otimes_{k_\infty}\mathbb{T}_z(k_\infty). $

The map $\exp_{\widetilde{E}} : \Lie_{\widetilde{E}}(\widetilde{L}_\infty)\longrightarrow \widetilde{E}(\widetilde{L}_\infty)$ can be restricted to an homomophism of $A[z]$-modules from $\Lie_{\widetilde{E}}(\mathbb{T}_z(L_\infty))$ to $\widetilde{E}(\mathbb{T}_z(L_\infty))$ that we still denote by  $\exp_{\widetilde{E}}$.

 Following Anglès, Ngo Dac and Tavares Ribeiro, we define  $$H(\widetilde{E}(M[z]))=\dfrac{\widetilde{E}(\mathbb{T}_z(L_\infty))}{\widetilde{E}(M[z])+\exp_{\widetilde{E}}(\Lie_{\widetilde{E}}(\mathbb{T}_z(L_\infty)))}.$$ 
 
\begin{proposition}
$H(\widetilde{E}(M[z]))$ is a finitely generated and torsion $\F_q[z]$-module.  Furthermore, $H(\widetilde{E}(\widetilde{M}))=\{0\}$. 
 \end{proposition}
\begin{proof}
For the first assertion, the proof is the same as \cite[Theorem 3.3]{ref19} with $M$ instead of $\mathcal{O}_L$. For the second, it is the same as \cite[Proposition 2]{ref3} with $t$-modules instead of Drinfeld modules. 
It follows from the fact that the $\F_q(z)$-modules generated by $\mathbb{T}_z(L_\infty)$ is dense in $\widetilde{L}_\infty$. 
\end{proof}

We define the module of $z$-units by  $$U(\widetilde{E}(M[z]))=\{x\in \Lie_{\widetilde{E}}(\mathbb{T}_z(L_\infty)), \exp_{\widetilde{E}}(x)\in \widetilde{E}(M[z])\}.$$ 
\begin{proposition}\label{engendre}\cite[Proposition 1]{ref3}  We have : 
\begin{itemize}
\item $U(\widetilde{E}(\widetilde{M}))$ is the $\F_q(z)$-vector space generated by $U(\widetilde{E}(M[z]))$.

\item $U(\widetilde{E}(M[z]))$ is a finitely generated $A[z]$-module. 
\end{itemize}
\end{proposition}

Following Anglès, Ngo Dac and Tavares Ribeiro in \cite{ref2} for Drinfeld modules and in \cite{ref19} for $t$-modules   we define the module of Stark units attached to a $t$-module.

\bigskip

We denote by  $\ev:\mathbb{T}_z(L_\infty)\longrightarrow L_\infty$ the evaluation at $z=1$. The {\it module of Stark units} of $M$ is defined by $$U_{St}(E(M))=\ev(U(\widetilde{E}(M[z]))).$$ 
As $\ev(\exp_{\widetilde{E}})=\exp_E$, we can see that $U_{St}(E(M))$ is an $A$-submodule of  $U(E(M)).$
\bigskip 

We define the morphism of  $\mathbb{F}_q[z]$-modules

$$\begin{array}{ccccc}
\alpha:	\mathbb{T}_z(L_\infty)^n &\to& H(\widetilde{E}(M[z])) \\
x &\mapsto &\dfrac{\exp_{\widetilde{E}}(x)-\exp_E(x)}{z-1} \mod (M[z]^n+\exp_{\widetilde{E}}(\mathbb{T}_z(L_\infty)^n)
\end{array}.$$

For $f\in A[z]$, we set $$H(\widetilde{E}(M[z]))[f]=\{x\in H(\widetilde{E}(M[z])), fx=0\}.$$

\begin{proposition}\label{iso}
We have an isomorphism of $A$-modules induced by $\alpha$ : 

$$\bar{\alpha}:\dfrac{U(E(M))}{U_{St}(E(M))}\cong H(\widetilde{E}(M[z]))[z-1].$$
\end{proposition}

\begin{proof}
It is the same as \cite[Theorem 3.3]{ref19}  replacing $\mathcal{O}_L$ by $M$.

\end{proof}

\begin{theorem}\label{fitting}
$\dfrac{U(E( M))}{U_{St}(E( M))}$ is a finite $A$-module and $U_{St}(E( M))$ is an $A$-lattice in $\Lie_E(L_\infty)$. Furthermore, $\left\vert\dfrac{U(E( M))}{U_{St}(E( M))}\right\vert_A=\vert H(E( M))\vert_A$.

\end{theorem}

\begin{proof}
It is the same as \cite[Theorem 3.3]{ref19}  replacing $\mathcal{O}_L$ by $M$. 

\end{proof}

With the class formula, as Anglès and Tavares Ribeiro in  \cite[Theorem 1]{ref3}  for Drinfeld modules and Anglès, Ngo Dac and Tavares Ribeiro  in  \cite[Theorem 3.3]{ref19} for $t$-modules,  we obtain $$\mathcal{L}(E(M))=[\Lie_E(M):U_{St}(E(M))]_A.$$

For a general $A$, it was proved by Mornev in \cite{ref21} for some Drinfeld modules by using shtuka cohomology and by Anglès, Ngo Dac and Tavares Ribeiro for some $t$-modules, in particular Drinfeld modules in \cite{ref1} by using the $z$-deformation. 

\section{Equivariant trace formula} ${}$\par
In this section, we recall the theory of nuclear operators and determinants which was introduced by Taelman  in \cite[Section 2]{ref8} and then developped for the equivariant setting by Ferrara, Green, Higgins and D. Popescu in \cite[Section 2]{ref4}. 

\subsection{Nuclear operators} ${}$\par

In what follows, $R_\ell=\ell[G]$ where $\ell=\F_q$ or $\F_q(z)$. All the proofs of this section for $\ell=\F_q$ are in \cite[Section 2]{ref4}. The same arguments work for $\ell=\F_q(z)$. 

Let $V$ be a $R_\ell$-module which is $R_\ell$-projective. 

We will look at the determinant of $V$ as a $R_\ell$-module and not as a  $l$-vector space. In this case, we look at determinant of continuous endomorphism of finitely generated $R_\ell$-modules  and projective.

\begin{definition} \label{base}
 
Let $\mathcal{U}=\{U_i\}_{i\geq m}$ be a sequence of open $R_\ell$-submodules of $V$ which verify : 

\begin{itemize}
\item for all $i$, $V/U_i$ is finitely generated,
\item every $U_i$ is $G$-cohomologically trivial,
\item $\forall i\geq m, U_{i+1}\subset U_i$,
\item $\mathcal{U}$ is a basis of neighborhoods of 0 in $V$. 
\end{itemize}
\end{definition}
 
First we assume that $\mathcal{U}$ exists and we define all that follow for $(V,\mathcal{U})$. Then we will see that it does not depend on the choice of $\mathcal{U}$.

 \begin{definition} 
 
Let $\phi$ be an endomorphism of $V$. We say that $\phi$ is {\it locally contracting} for  $\mathcal{U}$ if there exists $I\in \mathbb{N}$ such that $I\geq m$ and $\forall i\geq I,  \phi(U_i)\subset  U_{i+1} $. Such a neighborhood  $U=U_I$ of $0$ is called a {\it nucleus} for $\phi$. 
\end{definition}

In particular, if  $V$ is already a finitely generated $R_\ell$-module, we can take $U_i=\{0\}$ for $i\geq1$ and every endomorphism of $V$ is locally contracting.

\begin{proposition}
 
Let $\phi$ and $\psi$ be locally contracting endomorphisms of $V$ for $\mathcal{U}$. Then
\begin{itemize}
\item There exists a common nucleus. 
\item  $\phi+\psi$ is locally contracting. 
\item $\phi\circ\psi$ is locally contracting. 
\end{itemize}
\end{proposition}

We define the  $R_\ell[\![Z]\!]$-modules $$V[\![Z]\!]/Z^N=V\otimes_{R_\ell}R_\ell[\![Z]\!]/Z^N \hspace{1cm} \text{and} \hspace{1cm} V[\![Z]\!]=\lim\limits_{\overleftarrow{N}} V[\![Z]\!]/Z^N.$$ 

A continuous $R_\ell[\![Z]\!]$-linear morphism $\psi$ of $V[\![Z]\!]$ (respectively $R_\ell[\![Z]\!]/Z^N$-linear morphism of $V[\![Z]\!]/Z^N$) can be write as $\psi=\sum\limits_{r\geq 0}\phi_rZ^r$ (respectively $\psi=\sum\limits_{r\geq 0}^{N-1}\phi_rZ^r$) where the $\phi_r$ are endormorphisms of $V$.
Let  $V$ be a compact,  $R$-module $G$-cohomologically trivial. For $j\geq i\geq m$, and $U_i, U_j$ as in Definition \ref{base}, $V/U_i$ and $U_i/U_j$ are $G$-cohomologically trivial because $V, U_i$ and $U_j$ are. The $R_\ell$-modules $V/U_i$ and $U_i/U_j$ are finitely generated and projective. Therefore, we can take determinants of endomorphisms.

 \begin{definition} 
 
A linear $R_\ell[\![Z]\!]$-continuous endomorphism $\phi$ of $V[\![Z]\!]$ (respectively $V[\![Z]\!]/Z^N$) is said {\it nuclear} if for all $r\geq 0$ (respectively for all $r$ such that $N>r\geq 0$) the endomorphisms $\phi_r$ of $V$ are locally contracting. 
\end{definition}

\begin{proposition} \cite[Proposition 2.1.9]{ref4}
Let $\psi : V[\![Z]\!]/Z^N\longrightarrow V[\![Z]\!]/Z^N$ be a nuclear endomorphism. Let $U=U_J$ and $W=U_I$ be common  nuclei for $\phi_n$. Then 
 $$\det_{R_\ell[\![Z]\!]/Z^N}(1+\psi\vert V/U)=\det_{R_\ell[\![Z]\!]/Z^N}(1+\psi\vert V/W).$$ 
\end{proposition}

With this proposition, we see that the determinant does not depend on the choice of the nucleus. So we can have the next definition.
 \begin{definition} 
 
Let $\psi$ be a nuclear endomorphism of $V[\![Z]\!]/Z^N$ and $U$ be a common nucleus for $\phi_r$. Then we set $$\det_{R_\ell[\![Z]\!]/Z^N}(1+\psi\vert V)=\det_{R_\ell[\![Z]\!]/Z^N}(1+\psi\vert V/U). $$  Moreover, if $\psi$ is a nuclear endormophism of $V[\![Z]\!]$, we define the determinant of $(1+\psi)$ in $R_\ell[\![Z]\!]=\lim\limits_{\overleftarrow{N}} R_\ell[\![Z]\!]/Z^N$ as being $$\det_{R_\ell[\![Z]\!]}(1+\psi\vert V)=\lim_{\overleftarrow{N}} \det_{R_\ell[\![Z]\!]/Z^N}(1+\psi\vert V).$$ 
\end{definition}

\begin{proposition}\cite[Proposition 2.1.12]{ref4} \label{prod}
Let $\phi$ and $\psi$ be two nuclear endomorphisms of $V[\![Z]\!]$. Then $(1+\phi)(1+\psi)-1$ is nuclear and $$\det_{R_\ell[\![Z]\!]}((1+\psi)(1+\phi)\vert V)=\det_{R_\ell[\![Z]\!]}(1+\psi\vert V)\det_{R_\ell[\![Z]\!]}(1+\phi\vert V).$$
\end{proposition}

\begin{proposition} \cite[Proposition 2.1.13]{ref4}
Let $V'\subset V$ be a  closed $R_\ell$ sub-module of $V$ which is $G$-cohomologically trivial. We set $V''=V/V'$. Let $\mathcal{U'}=\{U'_i\}_{i}$, where $U'_i=U_i\cap V'$ and $\mathcal{U''}=\{U''_i\}_{i}$ where $U''_i$ is the image of $U_i$ in $V/V'$. We assume that $U'_i$ and $U''_i$ are $G$-cohomologically trivial. Let $\psi=\sum\phi_rZ^r:V[\![Z]\!]\longrightarrow V[\![Z]\!]$ be a nuclear endormorphism such that $\phi_r(V')\subset V' $ for all $r$. Then the endomorphisms induced by $\psi$ over $(V',\mathcal{U}')$ and $(V'',\mathcal{U}'')$ are nuclear. Furthermore $$\det_{R_\ell[\![Z]\!]}(1+\psi\vert V)=\det_{R_\ell[\![Z]\!]}(1+\psi\vert V')\det_{R_\ell[\![Z]\!]}(1+\psi\vert V'').$$
\end{proposition}

We assume that $V$ is a compact $R_\ell$-module $G$-cohomologically trivial. We will see that the determinant of $V$ does not depend of the choice of the basis.  Let $\mathcal{U}=\{U_i\}_{i}$ and $\mathcal{U'}=\{U'_i\}_{i}$ be two bases of $V$ as in Definition \ref{base}. Let $\psi\in End_{R_\ell}(V)$ and $\phi=\phi_rZ^r\in End_{R_\ell[\![Z]\!]}(V[\![Z]\!])$ be such that $\psi$ is locally contracting and  $\phi$ is nuclear with $\mathcal{U}$ and $\mathcal{U}'$.

 \begin{definition} \cite[Definition 2.2.1]{ref4}
 
We say that $\mathcal{U}$ {\it $\psi$-dominates} $\mathcal{U}'$  and we write it  $\mathcal{U} \succeq_\psi\mathcal{U}'$ if there exists $N\in \mathbb{N}$ such that for $i\geq N$,there exists $j\geq N$ such that $U_i\supseteq U'_j$ and $\psi(U_i)\subseteq U'_j$.

We say that $\mathcal{U}$ $\phi$-dominates $\mathcal{U}'$  and we write   $\mathcal{U} \succeq_\phi\mathcal{U}'$ if for all $r$, $\mathcal{U} \succeq_{\phi_r}\mathcal{U}'$.
\end{definition}

 \begin{Lem} 
Let $V,\psi$, $\mathcal{U}$ and $\mathcal{U}'$ be such that $\mathcal{U} \succeq_\psi\mathcal{U}'$. Then $$\det_{R_\ell[\![Z]\!]}(1+\psi\vert V)=\det'_{R_\ell[\![Z]\!]}(1+\psi\vert V).$$
\end{Lem}

\subsection{Some good bases} ${}$\par
We will see some examples of compact, projective $R_\ell$-modules and some associated bases $G$-cohomologically trivial that we could use.

Let $(M,\mathcal{W}^\infty)$ be an almost taming pair for $L/K$. For $P\in \Spec(A)$, we denote by $\widehat{A}_{P,\ell}$ the $P$-adique completion of $\ell[\theta]$ and $\widehat{\ell(\theta)}_P$ the $P$-adique completion of $\ell(\theta)$. We set $M_\ell=M\otimes_{\F_q}\ell$. It means that $M_\ell=M$ when $\ell=\F_q$ and $M_\ell=\widetilde{M}$ when $\ell=\F_q(z)$. In the same way,  $\mathcal{W}^\infty_\ell=\mathcal{W}_\ell^\infty\otimes_{\F_q}\ell$. We define for $i\geq 0$, 

$$U_{i,P,\ell}=M_\ell\otimes_{\ell[\theta]}P^i\widehat{A}_{P,\ell} \text{ and  } U_{i,\infty,\ell}=\mathcal{W}^\infty\otimes_{\ell(z)[\![\theta^{-1}]\!]}\theta^{-i}\ell[\![\theta^{-1}]\!].$$ 

These $U_{i,P,\ell}$ and $U_{i,\infty,\ell}$ are $G$-cohomologically trivial by choice of $M$ and $\mathcal{W}^\infty$ and are a decreasing sequence for inclusion. 

Let $V$ be a element ot the class $\mathcal{C}$ which correspond to $\ell[G]$-modules compact which verify an exact sequence $0\longrightarrow \Lie_E(L_{\infty,\ell})/\Lambda \longrightarrow V\longrightarrow H\longrightarrow0$ where $\Lambda$ is a $\ell[\theta]$-lattice of $\Lie_E(L_{\infty,\ell})$ and $H$ is a $\ell(\theta)$-vector space of finite dimension which is a $\ell[\theta][G]$-module. We want to construct a basis $\mathcal{U}$ of $V$. To do so, we will use some $R_\ell$-submodules  of $L_{\infty,\ell}$ which are  $G$-cohomologically trivial. 
For $i\geq 0$, we take $U_{i,\infty,\ell}\subseteq L_{\infty,\ell}$.

They are compact $R_\ell$-submodules, $G$-cohomologically trivial of $L_{\infty,\ell}$ which form a basis of neighborhoods of $0$ in $L_{\infty,\ell}$. 

We recall that $\Lie_E(L_{\infty,\ell})\cong L_{\infty,\ell}^n$ for some $n$. 
As $\Lambda$ is discrete, there exists $r\geq 0$ such that $(U_{r,\infty,\ell})^n\cap \Lambda=\{0\}$. For $i\geq r$, we associate $(U_{i,\infty})^n$ with its image in the exact sequence. If we define $\mathcal{U}=\{U_{i,\infty,\ell}\}_{i\geq r}$, $\mathcal{U}^n$ gives us a good basis of $V$. \bigskip

Let $S$ be a finite set of places of $k$ containing $\infty$. We set $V_\ell$ the $\ell(\theta)$-vector space spanned by $M_\ell$ ( i.e.,$V=M_\ell\otimes_{\ell[\theta]} \ell(\theta)$). For $P\in \Spec(\ell[\theta])$, we denote $V_{P,\ell}=V_\ell\otimes_{\ell(\theta)}\widehat{\ell(\theta)}_P$, $M_{P,\ell}=M_\ell\otimes_{\ell[\theta]}\widehat{\ell[\theta]}_P$. We denote $V_{S,\ell}=\prod\limits_{P\in S} V_{P,\ell}$ and $M_{S,\ell}=M_\ell\otimes_{\ell[\theta]}\ell[\theta]_S$ where $\ell[\theta]_S$ is the ring of $S$-integers (i.e., $ \ell[\theta]_{S}=\{a\in \ell(\theta) \vert  \forall v\not\in S, v(a)\geq 0 \}$).

We can show that $M_{S,\ell}$ is a lattice of  $V_{S,\ell}$. In particular, 
$M_{S,\ell}$ is discrete and co-compact in $V_{S,\ell}$. 

We see that $V_{S,\ell}$ is  $G$-cohomologically trivial because $V_{S,\ell}=\prod\limits_{P\in S} V_{P,\ell}$ and the  $V_{P,\ell}$ are.

As $M_\ell$ is $\ell[\theta][G]$-projective then $M_{S,\ell}$ is $\ell[\theta]_S[G]$-projective. Thus $M_{S,\ell}$ is $G$-cohomologically trivial which implies that $V_{S,\ell}/M_{S,\ell}$ too. 

So $V_{S,\ell}/M_{S,\ell}$ is a compact and projective $R_\ell$-module. As previously, we will take some $R_\ell$-modules $G$-cohomologically trivial of $V_{S,\ell}$ which will induce a basis over $V_{S,\ell}/M_{S,\ell}$. For $i\geq 0$ and $P\in S\cap \Spec(\ell[\theta])$, we set $U_{i,P,\ell}=M_\ell\otimes_{\ell[\theta]}P^i\widehat{\ell[\theta]}_P$.

For $i\geq 0$, we denote 
$U_{i,\infty,\ell}=\mathcal{W}^\infty_\ell
\otimes_{\ell[\![\theta^{-1}]\!]}\theta^{-i}\ell[\![\theta^{-1}]\!]$.
We have $$U_{i,S,\ell}=\prod\limits_{v\in S}U_{i,v,l}\subseteq \prod\limits_{v\in S}V_{v,l}=V_{S,\ell}.$$ 

$(U_{i,S,\ell})_{i\geq 0}$ form a basis of projective $R_\ell$-modules, open in $V_{S,\ell}$. As $M_{S,\ell}$ is discrete in $V_{S,\ell}$, there exists $r\in \mathbb{N}^*$ such that $U_{r,S,\ell}\cap M_{S}=\{0\}$. For $i\geq r$, we identify $U_{i,S,\ell}$ with their image in $V_{S,\ell}/M_{S,\ell}$. We set $\mathcal{U}=(U_{i,S})_{i\geq r}$. It gives us a good basis for $V_{S,\ell}/M_{S,\ell}$. It means that $\mathcal{U}^n$ is a good basis for $(V_{S,\ell}/M_{S,\ell})^n$.
\bigskip

First, we fix an almost taming pair $(M,\mathcal{W}^\infty)$ for $L/K$. Then we will see that the determinant does not depend on the choice of this pair. We set $\mathcal{O}_{K,S,\ell}=\mathcal{O}_{K,S}\otimes_{\F_q}\ell$. 
As we can use the same arguments that \cite[Section 2.32]{ref4} but with $(V_{S,\ell}/M_{S,\ell})^n$ instead of $L_S/M_S$, we just give the statements. 
 \begin{Lem} \label{oslocal}
Let $S$ be a finite set of places of $\ell(\theta)$ containing $\infty$. Let $\phi=a\tau^s$ for $a\in \mathcal{M}_n(\mathcal{O}_{K,S,\ell})$ and $s\geq 1$. Then $\phi$ is a locally contracting endomorphism of  $(V_{S,\ell}/M_{S,\ell})^n$  for the basis induced by $(M,\mathcal{W}^\infty)$.  
\end{Lem}

 \begin{Cor} \label{nucl}
Let $S$ be a finite set of places of $\ell(\theta)$ containing $\infty$. Every $\phi\in  \mathcal{M}_n(\mathcal{O}_{K,S,\ell})\{\tau\}\tau$ is a locally contracting endomorphism of $(V_{S,\ell}/M_{S,\ell})^n$ for the basis induced by $(M,\mathcal{W}^\infty)$.  Moreover, $\psi\in  \mathcal{M}_n(\mathcal{O}_{K,S,\ell})\{\tau\}\tau[\![Z]\!]$ is a nuclear endomorphism of $(V_{S,\ell}/M_{S,\ell})^n[\![Z]\!]$ for this basis. 
\end{Cor}

 \begin{proposition} \label{commut}
Let $S$ be a finite set of places of $\ell(\theta)$ containing $\infty$. Let $a,b\in \mathcal{M}_n(\mathcal{O}_{K,S,\ell})$ and $\phi=b\tau^s$ for $s\geq 1$. Then for all $m\geq 1$, $$\det_{R_\ell[\![Z]\!]}(1+a\phi Z^m\vert (V_{S,\ell}/M_{S,\ell})^n)=\det_{R_\ell[\![Z]\!]}(1+\phi aZ^m\vert (V_{S,\ell}/M_{S,\ell})^n).$$
\end{proposition}

\begin{Lem} 
Let $S$ be a finite set of places of $\ell(\theta)$ containing $\infty$. Let \newblock{ $\psi\in \mathcal{M}_n(\mathcal{O}_{K,S,\ell})\{\tau\}[\![Z]\!]\tau$} seen as a $R_\ell[\![Z]\!]$ endomorphism of $(V_{S,\ell}/M_{S,\ell})^n[\![Z]\!]$. Then $\det_{R_\ell[\![Z]\!]}(1+\psi \vert (V_{S,\ell}/M_{S,\ell})^n)$ is independent of the almost taming pair.
\end{Lem}

\subsection{The trace formula} ${}$\par

In this section, let $(M,\mathcal{W}^\infty$)
be an almost taming pair. This section was in \cite{ref4} when $\ell=\F_q$ and $E$ is a Drinfeld module.

\begin{Lem} \label{localisation2}
Let $S$ be a finite set of primes of  $\ell(\theta)$ including $\infty$, $P\in \MSpec(\ell[\theta])\setminus S$ and $S'=S\cup \{P\}$. Then for $f\in  \mathcal{M}_n(\mathcal{O}_{K,S,\ell})\{\tau\}[\![Z]\!]\tau Z$, we have $$\det_{R_\ell[\![Z]\!]}(1+f \vert (M_\ell/P M_\ell)^n)=\dfrac{\det_{R_\ell[\![Z]\!]}(1+f \vert (V_{S',\ell}/ M_{S',\ell})^n)}{\det_{R_\ell[\![Z]\!]}(1+f \vert (V_{S}/ M_{S,\ell})^n)}.$$
\end{Lem}
\begin{proof}
We recall that $V_{P,\ell}=V_\ell\otimes_{\ell(\theta)}\widehat{\ell(\theta)}_P$, $M_{P,\ell}=M_\ell\otimes_{\ell[\theta]}\widehat{\ell[\theta]}_P$ and $M_{S,\ell}=M_\ell\otimes_{\ell[\theta]}\ell[\theta]_S$.

We have the exact sequence 

$$0\longrightarrow M_{P,\ell}^n\longrightarrow^\psi(\dfrac{V_{S',\ell}}{M_{S',\ell}})^n\longrightarrow^\eta(\dfrac{V_{S,\ell}}{M_{S,\ell}})^n\longrightarrow0$$

where for $a=(a_1,\ldots, a_n)\in M_{P,\ell}^n$, $\psi=(\psi_1,\ldots, \psi_n)$,  $\psi_i(a_i)=\overline{(0,a_i)}\in \dfrac{V_{S',\ell}}{M_{S',\ell}}$.  As $\widehat{\ell(\theta)}_P=\ell[\theta]_{S'}+\widehat{\ell[\theta]}_P$, we have $V_{P}=M_{S',\ell}+M_{P,\ell}$. 

Thus, for $b\in V_{P}$, there exists $a'\in M_{S',\ell}$ and $b'\in M_{P,\ell}$ such that $b=a'+b'$. We define $\eta=(\eta_1,\ldots,\eta_n)$ such that for $a_i\in V_{S,\ell}$ et $b_i\in V_{P}$, $\eta_i(\overline{(a_i,b_i)})=\overline{a_i-a_i'}$. 
\\
First we show that $\eta$ is well defined. We write $b_i=a_i'+b_i'=a_i''+b_i''$ where $a_i'\in M_{S',\ell}$ and $b_i'\in M_{P,\ell}$. So we have $a_i'-a_i''=b_i''-b_i'\in M_{S',\ell}\cap M_{P,\ell}=M_{S,\ell}$ as $\ell[\theta]_{S'}\cap \widehat{\ell[\theta]}_P=\ell[\theta]_{S}$.
For $b\in M_{P,\ell}^n$, we have $\eta(\psi(b))=0$. 

Let $(a_i,b_i) \mod M_{S',\ell} \in \Ker(\eta_i)$. There exists  $a_i'\in M_{S',\ell}$, $b_i'\in M_{P,\ell}$ and $c_i\in M_{S,\ell}$ such that $b_i=a_i'+b_i'$ et $a_i-a_i'=c_i$.

As $a_i'\in M_{S',\ell}$, $b_i'\in M_{P,\ell}$ and $c_i\in M_{S,\ell}$, we have $$(a_i,b_i)-(a_i,a_i)=(0,b_i-a_i)=(0,b_i-a_i'-c_i)=(0,b_i'-c_i)\in \{0\}\times M_{P,\ell}.$$ Thus $\Ker(\eta_i)\subset Im(\psi_i)$.
\\

We have the following open neighborhoods of 0 : $\mathcal{U}$ and $\mathcal{U}'$  over  $V_{S,\ell}/M_{S,\ell}$  and $V_{S',\ell}/M_{S',\ell}$ induced par the almost taming pair : $(M,\mathcal{W}^\infty$).

We have $\mathcal{U}=\prod\limits_{v\in S}\mathcal{U}_v$ and $\mathcal{U'}=\prod\limits_{v\in S'}\mathcal{U}_v$ where $\mathcal{U}_v=\{M_\ell\otimes_{\ell[\theta]} v^i\widehat{\ell[\theta]}_v\}_{i\geq 1}$ for $v\in \Spec(A)$ and   $\mathcal{U}_\infty=\{\mathcal{W}_\ell\otimes_{\ell[\![\theta^{-1}]\!]}\theta^{-i}\ell[\![\theta^{-1}]\!]\}_{i\geq 1}$. We have $\eta(\mathcal{U}'^n)=\mathcal{U}^n$ and $\psi^{-1}(\mathcal{ U}'^n)=\mathcal{U}_P^n$.

As $f\in  \mathcal{M}_n(\mathcal{O}_{K,S,\ell})\{\tau\}[\![Z]\!]\tau Z$,  $f_r\in   \mathcal{M}_n(\mathcal{O}_{K,S,\ell})\{\tau\}\tau$. Thus they commute with $\psi$ and $\eta$ and are locally contracting for $\mathcal{U}_{P}^n, \mathcal{U}^n$ et $\mathcal{U}'^n$.

We obtain $$\det_{R_\ell[\![Z]\!]}(1+f \vert (V_{S',\ell}/M_{S',\ell})^n)=\det_{R_\ell[\![Z]\!]}(1+f\vert M_{P,\ell}^n)\det_{R_\ell[\![Z]\!]}(1+\phi \vert (V_{S,\ell}/M_{S,\ell})^n).$$
As $f\in   \mathcal{M}_n(\mathcal{O}_{K,S,\ell})\{\tau\}[\![Z]\!]\tau Z$ and $P\not\in S$, $f_r(PM_{P,\ell}^n)\subset PM_{P,\ell}^n$. It implies that we can take $PM_{P,\ell}^n$ as commun nucleus for $f_r$. Furthermore, $M_{P,\ell}/PM_{P,\ell}\cong M_\ell/{P}M_\ell$. It follows $$\det_{R_\ell[\![Z]\!]}(1+f \vert M_{P,\ell}^n)=\det_{R_\ell[\![Z]\!]}(1+f \vert (M_\ell/{P}M_\ell)^n) =\dfrac{\det_{R_\ell[\![Z]\!]}(1+f \vert (V_{S',\ell}/ M_{S',\ell})^n)}{\det_{R_\ell[\![Z]\!]}(1+f \vert (V_{S,\ell}/ M_{S,\ell})^n)}.$$
\end{proof}

The next theorem is the same as \cite[Proposition 3.5]{ref16} but for $\ell[G]$ instead of $l$ or $\ell[G]$ instead of $\F_q[G]$ for \cite[Theorem 3.0.2]{ref4}. We recall it for the convenience of the reader. 

\begin{theorem} \label{formuledeclas2}
Let $S$ be a finite set of places of $\ell(\theta)$ containing  $\infty$, $P\in Spec(\ell[\theta])\setminus S$ and  $S'=S\cup \{P\}$ and $\psi\in  \mathcal{M}_n(\mathcal{O}_{K,S,\ell})\{\tau\}[\![Z]\!]\tau Z$. We have $$\prod\limits_{v\in \MSpec( \ell[\theta]_S)}\det_{R_\ell[\![Z]\!]}(1+\psi \vert ( M_\ell/v M_\ell)^n)=\det_{R_\ell[Z]\!]}(1+\psi \vert (V_{S,\ell}/ M_{S,\ell})^n)^{-1}.$$
\end{theorem}
\begin{proof}

Let $\psi=\sum\limits_{r=1}^\infty\psi_rZ^r\in\mathcal{M}_n(\mathcal{O}_{K,S,\ell})\{\tau\}[\![Z]\!]\tau Z$. We want to show that we have the following equality
$$\prod\limits_{v\in \MSpec( \ell[\theta]_S)}\det_{R_\ell[\![Z]\!]/Z^N}(1+\psi \vert  (M_\ell/v M_\ell)^n)=\det_{R_\ell[\![Z]\!]/Z^N}(1+\psi \vert (V_{S}/ M_{S,\ell})^n)^{-1}.$$

Let $D=D_N$  be such that $\deg_\tau\psi_r<\dfrac{rD}{N}$ for all $r<N$. We set $$T=T_D=S\cup \{v\in \MSpec( \ell[\theta]_S)\hspace{0.1cm}\vert\hspace{0.1cm} \forall w\vert v, [\mathcal{O}_{K,S,\ell}/w:\ell]<D\}.$$

By applying succesively the previous lemma, we get $$\prod\limits_{v\in T\setminus S}\det_{R_\ell[\![Z]\!]/Z^N}(1+\psi \vert  (M_\ell/v M_\ell)^n)=\dfrac{\det_{R_\ell[\![Z]\!]/Z^N}(1+\psi (\vert V_{T,\ell}/ M_{T,\ell})^n)}{\det_{R_\ell[\![Z]\!]/Z^N}(1+\psi \vert (V_{S,\ell}/ M_{S,\ell})^n)} .$$ Furthermore, $\MSpec( \ell[\theta]_S)=T\setminus S\cup \MSpec( \ell[\theta]_T)$. So we obtain
  
$$\prod\limits_{v\in \Spec( \ell[\theta]_S)}X_v= \prod\limits_{v\in \Spec( \ell[\theta]_T)}X_v\prod\limits_{v\in T\setminus S}X_v
\hspace{0,1cm }
\text{where}\hspace{0,1cm } X_v=\det_{R_\ell[\![Z]\!]/Z^N}(1+\psi \vert  (M_\ell/v M_\ell)^n).$$
It suffices to show $$\prod\limits_{v\in \MSpec(\ell[\theta]_T)}\det_{R_\ell[\![Z]\!]/Z^N}(1+\psi \vert  (M_{\ell}/v M_{\ell})^n)=\det_{R_\ell[\![Z]\!]/Z^N}(1+\psi \vert (V_T/ M_{T,\ell})^n)^{-1}.$$

We set $\mathcal{S}_{D,N}\subset \mathcal{M}_n(\mathcal{O}_{K,T,\ell})\{\tau\}[\![Z]\!]/Z^N$ as $$\mathcal{S}_{D,N}=\left\{1+\sum\limits_{r=1}^{N-1}\psi_rZ^r \mod Z^N\vert \deg_\tau(\psi_r)<\dfrac{rD}{N}\forall r<N\right\}.$$

It is a multiplicative group where $(1+\psi)\mod Z^N$ belongs to by the choice of $D$.

We choose $T$ so that $\mathcal{O}_{K,T,\ell}$ does not have any residual group of degree  $d<D$ over $l$. So for $d<D$, there exists $M_d\in \N$,  $f_{dj},a_{dj}\in \mathcal{O}_{K,T,\ell}$ for $1\leq j\leq M_d$ such that $1=\sum\limits_{j=1}^{M_d}f_{dj}\left(a_{dj}^{q^d}-a_{dj}\right)$. 
 
Indeed, for $d<D,$ we denote by $I_d$ the ideal of  $\mathcal{O}_{K,T,\ell}$ spanned by  $\{a^{q^d}-a, a\in \mathcal{O}_{K,T,\ell}\}$. If we assume $I_d\not=\mathcal{O}_{K,T,\ell} $, there exists a maximal ideal $m_d$ such that $I_d\subset m_d\varsubsetneq \mathcal{O}_{K,T,\ell}$. 
  So for all $a\in \mathcal{O}_{K,T,\ell}$, $a^{q^d}-a\equiv 0 \mod I_d$ and so $a^{q^d}-a\equiv 0 \mod m_d$. 
  It implies that $\dim_{\ell}\mathcal{O}_{K,T,\ell}/m_d\leq d$ and we get a contradiction by the choice of $T$. It means that $I_d=\mathcal{O}_{K,T,\ell}$ and there exist $M_d\in \N$,  $f_{dj},a_{dj}\in \mathcal{O}_{K,T,\ell}$ pour $1\leq j\leq M_d$ such that $1=\sum\limits_{j=1}^{M_d}f_{dj}(a_{dj}^{q^d}-a_{dj})$. 
  
We set $F_{dj}=f_{dj}I_n$ and $A_{dj}=a_{dj}I_n$. 
Let $B\in\mathcal{M}_n(\mathcal{O}_{K,T,\ell})$, $ r<N$ and  $d<D$. We have $1-B\tau^dZ^r=1-\sum\limits_{j=1}^{M_d}BF_{dj}(\tau^dA_{dj}-A_{dj}\tau^d)Z^r$, i.e., $$1-B\tau^dZ^r-\sum\limits_{j=1}^{M_d}BF_{dj}A_{dj}\tau^dZ^r=1-\sum\limits_{j=1}^{M_d}BF_{dj}\tau^dA_{dj}Z^r.$$ 

Moreover, $1-B\tau^dZ^d-\sum\limits_{j=1}^{M_d}BF_{dj}A_{dj}\tau^dZ^r\equiv (1-B\tau^dZ^r)(\prod\limits_{j=1}^{M_d}1-A_{dj}(BF_{dj}\tau^d)Z^r)\mod Z^{r+1}$ and $1-\sum\limits_{j=1}^{M_d}BF_{dj}\tau^dA_{dj}Z^r\equiv \prod\limits_{j=1}^{M_d}1-(BF_{dj}\tau^d)A_{dj}Z^r\mod Z^{r+1}$.

So $1-B\tau^dZ^r\equiv\prod\limits_{j=1}^{M_d}\dfrac{1-(RF_{dj}\tau^d)A_{dj}Z^r}{1-A_{dj}(BF_{dj}\tau^d)Z^r}\mod Z^{r+1}$.

It means that $\left\{\dfrac{1-(S\tau^d)AZ^r}{1-A(S\tau^d)Z^r} \mid A,S\in \mathcal{M}_n(\mathcal{O}_{K,T,\ell})\right\}$ is a system of generators of $\mathcal{S}_{D,N}$. 

We have  $\det_{R_\ell[\![Z]\!]/Z^N}\left(\dfrac{1-(S\tau^d)AZ^r}{1-A(S\tau^d)Z^r} \vert  (M_{\ell}/v M_{\ell})^n\right)=1$ for $v\in \MSpec(\ell[\theta]_T)$.

By Proposition \ref{commut},  
$\det_{R_\ell[\![Z]\!]/Z^N}\left(\dfrac{1-(S\tau^d)AZ^r}{1-A(S\tau^d)Z^r} \vert (\dfrac{V_{T,\ell}}{ M_{T,\ell}})^n\right)=1$. 
So we obtain  $$\prod\limits_{v\in \MSpec(\ell[\theta]_T)}\det_{R_\ell[\![Z]\!]/Z^N}(1+\psi \vert  (M_{\ell}/v M_{\ell})^n)=1=\det_{R_\ell[\![Z]\!]/Z^N}(1+\psi \vert (V_{T,\ell}/ M_{T,\ell})^n)^{-1}$$ which gives us the desired result. 
\end{proof}

We define $E_\ell$ as $E$ if $\ell=\F_q$ and $\widetilde{E}$ if $\ell=\F_q(z)$.

We define 
$\psi_\ell$ as  $\dfrac{1-\phi_E(\theta) Z}{1-\delta_E(\theta) Z}-1$ if $\ell=\F_q$ and $\dfrac{1-\phi_{\widetilde{E}}(\theta) Z}{1-\delta_E(\theta) Z}-1$ if $\ell=\F_q(z)$.
\begin{Cor} \label{fonctionl2}
$\psi_\ell$ is a nuclear operator on  $(L_{\infty,\ell}/M_\ell)^n[\![Z]\!]$ and $$\det_{R_\ell[\![Z]\!]}(1+\psi_\ell \vert (L_{\infty,\ell}/M_\ell)^n)=\prod\limits_{v\in \MSpec(\ell[\theta])}\dfrac{\vert \Lie_{E_\ell}(M_\ell/v)\vert_G }{\vert E_\ell(M_\ell/v)\vert_G}.$$
\end{Cor}
\begin{proof}

We can easily see that 
$\psi_\ell =\sum\limits_{r\geq1}(\delta_E(\theta)-\phi_{E_\ell}(\theta) )Z^r\delta_E^{r-1}(\theta).$
Thus $\psi_\ell\in \mathcal{M}_n(\mathcal{O}_{K,\ell})\{\tau\}[\![Z]\!]\tau Z$ and we can apply Corollary \ref{nucl} to obtain that $\psi_\ell$ is a nuclear operator on  $(L_{\infty,\ell}/M_\ell)^n$ and  $(M_\ell/v)^n$ for all $v$.

We apply the previous theorem at the case $S=\{\infty\}$. As $V_{\infty,\ell}=L_{\infty,\ell}$ and for all $v\in\MSpec({\ell[\theta]}),\dfrac{\det_{R_\ell[\![Z]\!]}(1-\delta_E(\theta) Z \vert (M_\ell/vM_\ell)^n)}{\det_{R_\ell[\![Z]\!]}(1-\phi_{E_\ell}(\theta)Z \vert (M_\ell/vM_\ell)^n)}=\dfrac{\vert \Lie_{E_\ell}(M_\ell/vM_\ell)\vert_{G}}{\vert E_\ell(M_\ell/vM_\ell)\vert_{G}}$   it gives us 

\begin{equation*}
\begin{aligned} 
\det_{R_\ell[\![Z]\!]}(1-\psi_\ell \vert (L_{\infty,\ell}/M_\ell)^n) &=\prod\limits_{v\in \MSpec({\ell[\theta]})}\det_{R_\ell[\![Z]\!]}(1-\psi_\ell\vert (M_\ell/vM_\ell)^n)^{-1}\\
& =\prod\limits_{v\in \MSpec(\ell[\theta])}\dfrac{\vert \Lie_{E_\ell}(M_\ell/vM_\ell)\vert_{G}}{\vert E_\ell(M_\ell/vM_\ell)\vert_{G}}.
\end{aligned}
\end{equation*} 

\end{proof}
 
In particular, it allows us to have the next definition.

\begin{definition} 
We define the {\it $G$-equivariant $L$-functions} 
$$\mathcal{L}_G(E_\ell( M_\ell))=\prod\limits_{v\in \MSpec(\ell[\theta])}\dfrac{\vert \Lie_{E_\ell}(M_\ell/v)\vert_G }{\vert E_\ell(M_\ell/v)\vert_G}. $$ 
\end{definition}

\section{Volume and applications}
In this section, we recall the statements of \cite[Section 5]{ref4} but for $l$ instead of $\F_q$ and also for Anderson modules.
The arguments stay the same but there are some technical changes in the conditions of \ref{52} and \ref{53} because of the fact that $\theta$ and $\delta_E(\theta)$ can have different norms . 

\subsection{Volume}

Let $\Lambda$, $\Lambda'$ be two free $\ell[\theta][G]$-lattices with $\mathcal{B}$ and $\mathcal{B}'$ as bases. By definition, they are $\ell\left(\left(\theta^{-1}\right)\right)[G]$-bases for $L_{\infty,\ell}^n$. There exists $X\in Gl_{mn}(\ell\left(\left(\theta^{-1}\right)\right)[G])$ ($m=[K:k]$) such that $\mathcal{B}'=X\mathcal{B}$. Then $\det(X)$ depends on the choice of $\mathcal{B}$ and $\mathcal{B}'$ but not $\det(X)^+$ which is the image of $\det(X)$ by $$\ell\left(\left(\theta^{-1}\right)\right)[G]^\times\twoheadrightarrow \ell\left(\left(\theta^{-1}\right)\right)[G]^\times/\ell[\theta][G]^\times\cong \ell\left(\left(\theta^{-1}\right)\right)[G]^+$$ by \ref{unitaire}. 

\begin{definition} For  $\Lambda$ and $\Lambda'$ two free $\ell[\theta][G]$-lattices, we define $[\Lambda:\Lambda']_G=\det(X)^+$. 
\end{definition}

If $\Lambda$ and $\Lambda'$ are free $\ell[\theta][G]$-lattices such that $\Lambda'\subset \Lambda$ then $\Lambda/\Lambda'$ is a $\ell[\theta][G]$-module  $G$-cohomologically trivial which is a $l$-vector space of finite dimension. 
We can show that $[\Lambda:\Lambda']_G=\vert\Lambda/\Lambda'\vert_G$ . Morerover, if $\Lambda'', \Lambda'$ and $\Lambda$ are free $\ell[\theta][G]$-lattices, we can easily show that $[\Lambda:\Lambda'']_G=[\Lambda:\Lambda']_G[\Lambda':\Lambda'']_G$.

We want to extend this indice to projective $\ell[\theta][G]$-lattices. To do so, we need the following lemma. 
\begin{Lem}\label{reseaulibre}
Let $\Lambda$ be a projective $\ell[\theta][G]$-lattice of $L_{\infty,\ell}^n$. Then
\begin{itemize}
\item There exists a free $\ell[\theta][G]$-lattice $F$ of $L_{\infty,\ell}^n$ such that $\Lambda\subset F$, 
\item For such a  $F$, $F/\Lambda$ is a $l$-vector space of finite dimension and a $\ell[\theta][G]$-module  $G$-cohomologically trivial.

\end{itemize}

\end{Lem}

By Proposition \ref{fitt}, for such a $F$, $\Fitt_{\ell[\theta][G]}\left(\dfrac{F}{\Lambda}\right)$ is principal and admits a unique generator denoted  $\vert F/\Lambda\vert_G\in \ell\left(\left(\theta^{-1}\right)\right)[G]^+$ which is invertible.

This allows us to have the next definition. 
\begin{definition}
Let $\Lambda, \Lambda'$ be two projective $\ell[\theta][G]$-lattices of $L_{\infty,\ell}^n$. We take  two free $\ell[\theta][G]$-lattices $F$ and $F'$ of $L_{\infty,\ell}^n$ such that $\Lambda\subset F$ and $\Lambda'\subset F'$. We define $$[\Lambda:\Lambda']_G=[F:F']_G\dfrac{\vert F'/\Lambda'\vert_G}{\vert F/\Lambda\vert_G}.$$ 
\end{definition}

We can easily show that this definition is independent of the choice of $F$ and $F'$. 

As for the free lattices we have 
\begin{itemize}

\item if $\Lambda, \Lambda'$ and $\Lambda''$ are projective $\ell[\theta][G]$-lattices then $[\Lambda:\Lambda'']_G=[\Lambda:\Lambda']_G[\Lambda':\Lambda'']_G$.
\item if $\Lambda'\subset \Lambda$ are projective $\ell[\theta][G]$-lattices then $[\Lambda:\Lambda']_G=\vert\Lambda/\Lambda'\vert_G$. 
\end{itemize}
\bigskip

 Following \cite{ref4}, we define the class $\mathcal{C}$ of compact $\ell[\theta][G]$-modules  $V$ which are $G$-cohomologically trivial and verify an exact sequence of $\ell[\theta][G]$-modules 
$$0\longrightarrow L_{\infty,\ell}^n/\Lambda\xrightarrow[]{\hspace{0,2cm} f\hspace{0,2cm}} V\xrightarrow[]{\hspace{0,2cm}\pi\hspace{0,2cm}} H\longrightarrow 0$$ where $\Lambda$ is a $\ell[\theta][G]$-lattice of $L_{\infty,\ell}^n$ and $H$ is a   $\ell[\theta][G]$-module which is a $l$-vector space of finite dimension.

 $L_{\infty,\ell}^n/\Lambda$ is $\ell[\theta]$ divisible thus $\ell[\theta]$-injective as $\ell\left(\left(\theta^{-1}\right)\right)/\ell[\theta]$ is. It implies the existence of a section in the category of $\ell[\theta]$-modules. Thus we have the following isomorphism of $\ell[\theta]$-modules induced by  $(f,\id)$ : $$L_{\infty,\ell}^n/\Lambda\times s(H)\cong V .$$ 

We consider the structure of  $L_{\infty,\ell}^n/\Lambda\times s(H)$ as a  $\ell[\theta][G]$-module. For  $g\in G$ and $(x,s(h))\in~L_{\infty,\ell}^n/\Lambda\times~s(H)$, we set $g.(x,s(h)=(g.x+a_{g,h};b_{g,h})$ where $(a_{g,h};b_{g,h})$  is associated to $g.s(h)$ by the isomorphism $L_{\infty,\ell}^n/\Lambda\times s(H)\cong V$. We can show that $b_{g,h}=s(gh)$ and for $g_1, g_2\in G$ and $h\in H$, we have $a_{g_1g_2,h}=g_1a_{g_2,h}+a_{g_1,g_2h}$. 
\\

With this structure, we can introduce the next definition.

\begin{definition}
A $\ell[\theta][G]$- lattice $\Lambda'$  of $L_{\infty,\ell}^n$ is called  {\it $(V,\Lambda, H, s)$-admissible} if it verifies the following conditions :  
\begin{itemize}
\item $\Lambda\subset \Lambda'$,
\item $\Lambda'$ is a projective $\ell[\theta][G]$-module, 
\item $\Lambda'/\Lambda\times s(H)$ is a $\ell[\theta][G]$-submodule of $V$. 
\end{itemize}

\end{definition}
A $\ell[\theta][G]$-lattice $\Lambda'$  is said $V$-admissible if it is $(V,\Lambda, H, s)$-admissible for some $s$.

\begin{proposition}\label{admissible1}
For such  $(V,\Lambda, H, s)$, there exist lattices  $(V,\Lambda, H, s)$-admissible which are free $\ell[\theta][G]$-modules. 
\end{proposition}

Let $\Lambda'$ be an admissible $(V,\Lambda, H, s)$-lattice. Thus, there exists an exact sequence of $\ell[\theta][G]$-modules $$0\longrightarrow \Lambda'/\Lambda\times s(H)\longrightarrow V\longrightarrow L_{\infty,\ell}^n/\Lambda'\longrightarrow0.$$ As $V$ and $L_{\infty,\ell}^n/\Lambda'$ are $G$-cohomologically trivial, then $\Lambda'/\Lambda\times s(H)$ too. As it is a $\ell[\theta][G]$-module which is a $l$-vector space of finite dimension, by Proposition \ref{fitt} $\vert\Lambda'/\Lambda\times s(H)\vert_G$ is well defined and belongs to $\ell((\theta^{-1	}))[G]^+$. 

From now on we fix a projective $\ell[\theta][G]$-lattice $\Lambda_0$ to define the volume. This choice is not involved in the quotient of volumes which interrests us. 

\begin{proposition}
Let $V\in\mathcal{C}$ be such that $0\longrightarrow \Lie_E(L_{\infty,\ell})/\Lambda\longrightarrow V\longrightarrow H\longrightarrow 0$.
\begin{itemize}

\item Let $s$ be a section for $V$ and $\Lambda_1, \Lambda_2$ two admissible $(V,\Lambda, H, s)$-lattices. We have $$\dfrac{\vert \Lambda_1/\Lambda\times s(H)\vert_G}{[\Lambda_1:\Lambda_0]_G}=\dfrac{\vert \Lambda'_2/\Lambda\times s(H)\vert_G}{[\Lambda_2:\Lambda_0]_G}.$$
\item Let $s_1, s_2$ be two sections of the exact sequence and $\Lambda'$ an admissible \sloppy 
$(V,\Lambda, H, s_1)$ and $(V,\Lambda, H, s_2)$-lattice. 
Then we have $$\dfrac{\vert \Lambda'/\Lambda\times s_1(H)\vert_G}{[\Lambda':\Lambda_0]_G}=\dfrac{\vert \Lambda'/\Lambda\times s_2(H)\vert_G}{[\Lambda':\Lambda_0]_G}.$$
\end{itemize}
\end{proposition}

\begin{definition}
Let $V\in\mathcal{C}$ be such that $0\longrightarrow \Lie_E(L_{\infty,\ell})/\Lambda\rightarrow V\longrightarrow H\longrightarrow 0$. Let $s$ be a section and $\Lambda'$ an admissible $(V,\Lambda, H, s)$-lattice. We define $$\Vol_{\Lambda_0}(V)=\dfrac{\vert \Lambda'/\Lambda\times s(H)\vert_G}{[\Lambda':\Lambda_0]_G}.$$ By the previous proposition,  $\Vol_{\Lambda_0}(M)$ is independent of the section and the admissible  $(V,\Lambda, H, s)$-lattice. 
\end{definition}

\begin{proposition}\label{volume}
The funtion $\Vol_{\Lambda_0}:\mathcal{C}\longrightarrow \ell\left(\left(\theta^{-1}\right)\right)$ verifies : 
\begin{itemize}

\item $\Vol_{\Lambda_0}(L_{\infty,\ell}^n/\Lambda_0)=1$. 
\item If $V_1,V_2\in\mathcal{C}$, $\dfrac{\Vol_{\Lambda_0}(V_1)}{\Vol_{\Lambda_0}(V_2)}$ is independent of the choice of $\Lambda_0$. 
\item Let $V, V'\in \mathcal{C}$ be such that the diagram commutes

$$\begin{tikzcd}
0 \arrow[r] 
    & L_{\infty,\ell}^n/\Lambda\arrow[d ,]\arrow[r] \arrow[d]
    & V \arrow[d,"\phi" ]\arrow[d]\arrow[r,"\pi"] \arrow[d]
    & H \arrow[d,"id" ] \arrow[r] \arrow[d]
    & 0  \\
  0 \arrow[r ]
& L_{\infty,\ell}^n/\Lambda\arrow[r ]
& V'\arrow[r,"\pi'" ]
& H \arrow[r ]
& 0.\end{tikzcd}$$

Then $\Vol_{\Lambda_0}(M)=\Vol_{\Lambda_0}(M')$. 
\end{itemize}
\end{proposition}

\subsection{Application tangent to the identity}\label{52} ${}$\par
This section is inspired of \cite{ref4} and \cite{ref17}.
We endow $L_{\infty,\ell}$ with the norm sup of the local norms. We note it  $\vert\vert.\vert\vert$ and it is normalized so that $\vert\vert \theta\vert\vert=q$.
We still denote $\vert\vert.\vert\vert$ for $L_{\infty,\ell}^n$. We set $\mathcal{O}_{L_\infty,\ell}$ the elements of $L_{\infty,\ell}$ of norm less than 1.  The norm is taken so that $\mathcal{O}_{L_\infty,\ell}$ is a $R_\ell$-module (under the $G$-action, the norm is still less than 1 as it is the sup).

Let $c,d \in N$ be such that $\max_{i\in [\![0;q^n-1]\!]}\vert\vert\delta_E(\theta)^i\vert\vert=q^c$  and  $\vert\vert \delta_E(\theta)\vert\vert=q^d$.

Let $V_1, V_2\in \mathcal{C}$ verify $0\longrightarrow \Lie_E(L_{\infty,\ell})/\Lambda_j\xrightarrow{\hspace{0,1cm}\iota_j\hspace{0,1cm}} V_j\xrightarrow{\hspace{0,1cm}\pi_j\hspace{0,1cm}} H_j\longrightarrow 0$ for $j=1,2$.

 Let $r$ be large enough so that $(\theta^{-i}\mathcal{O}_{L_{\infty,\ell}})^n\cap\Lambda_j=\{0\}$ for $i\geq r$ et $j=1,2$. We identify $(\theta^{-i}\mathcal{O}_{L_{\infty,\ell}})^n$ with its image in $\Lie_E(L_{\infty,\ell})/\Lambda_j$. 
We fix a $\infty$-taming module $\mathcal{W}^\infty$ for $L_\ell/K_\ell$. Previously, we saw that $\{\iota_s(\mathcal{U}_{i,\infty})^n\}_{i\geq r}$ is a basis of neighborhoods of 0 in $V_j$ which are projective $R_\ell$-modules where $\mathcal{U}_{i,\infty}=\theta^{-i}\mathcal{W}^\infty$. For $a$ large enough and $i\geq r$, $\theta^{-a-i}\mathcal{O}_{L_\infty,\ell}\subset \mathcal{U}_{i,\infty}\subset \theta^{-i}\mathcal{O}_{L_\infty,\ell}$. 
We endow $\iota_j(\theta^{-i}\mathcal{O}_{L_\infty,\ell})$ with the norm so that $\iota_j:\theta^{-i}\mathcal{O}_{L_\infty,\ell}\longrightarrow \iota_j(\theta^{-i}\mathcal{O}_{L_\infty,\ell}$) is a bijective isometry for $j=1,2$. 

\begin{definition}\label{tangence}
Let $N\in\mathbb{N}$. We say that a continuous $R_\ell$-morphism  $\gamma:V_1\longrightarrow V_2$ is  {\it $N$-tangent to the identity} if theres exists $i\geq r$ such that

\begin{itemize}
\item $\gamma$ induces a bijective isometry $\gamma_i=(\iota_2^{-1}\circ\gamma\circ\iota_1):\theta^{-i}\mathcal{O}_{L_\infty,\ell}^n\longrightarrow \theta^{-i}\mathcal{O}_{L_\infty,\ell}^n$,
\item $\vert\vert \gamma_i(x)-x\vert\vert \leq \vert\vert \theta\vert \vert^{-N-a}\vert\vert x\vert\vert$ for all $x\in (\theta^{-i}\mathcal{O}_{L_\infty,\ell})^n$. 
\end{itemize}

If $\gamma$ is $N$-tangent to the identity for all $N\geq 0$, then we say that $\gamma$ is {\it $\infty$-tangent to the identity}. 
\end{definition}

\begin{proposition}\label{infiniment}
Let $\Gamma:L_{\infty,\ell}^n\longrightarrow L_{\infty,\ell}^n$ be a $R_\ell$-linear application given by the always convergent serie which belongs to $I_n+\tau \mathcal{M}_n(L_{\infty,\ell})\{\tau\}$. We assume that $\Gamma(\Lambda_1)\subset \Lambda_2$. We denote by $\bar{\Gamma}:(L_{\infty,\ell})^n/\Lambda_1\longrightarrow (L_{\infty,\ell})^n/\Lambda_2$ the induced application. We assume that $\gamma:V_1\longrightarrow V_2$ is a $R_\ell$-linear morphism such that $\iota_2^{-1}\circ\gamma\circ\iota_1=\bar{\Gamma}$ over $(\theta^{-l}\mathcal{O}_{L_\infty,\ell})^n$. Then $\gamma$ is $\infty$-tangent to the identity. 
\end{proposition}

In particular, $\exp_{\widetilde{E}}$ is $\infty$-tangent to the identity. 
\bigskip

Let $V_1,V_2\in\mathcal{C}$ and $\gamma:V_1\longrightarrow V_2$ be a $R_\ell$-linear isomorphism. We denote by  $\Delta_\gamma$ the  isomorphism of $V_1[\![Z]\!]$ such that $$\Delta_\gamma=\dfrac{1-\gamma^{-1}\delta_E(\theta)\gamma Z}{1-\delta_E(\theta) Z}-1=\sum\limits_{s\geq 1}\delta_sZ^{s}$$ where $\delta_s=(\delta_E(\theta)-\gamma^{-1}\delta_E(\theta).\gamma)\delta^{s-1}_E(\theta)$ for $s\geq 1$.

We recall that $\displaystyle\max_{i=0, \ldots, q^n-1} \vert\vert \delta_E^i(t)\vert\vert=q^{c}.$ 

\begin{Lem}\label{tangentmod}
Let $\gamma:V_1\longrightarrow V_2$ be a $R_\ell$-linear isomorphism  $N'$-tangent to the identity. If $N\in \N$ verifies $\forall m <N, N'-q^n\lfloor\dfrac{m}{q^n}\rfloor-c\geq 1 \text{ and } N'-q^n\lfloor\dfrac{m-1}{q^n}\rfloor-2c\geq 1$,  then $\Delta_\gamma \mod Z^N$ is a nuclear endomorphism of $V_1[\![Z]\!]/Z^n$. 

If $\gamma:V_1\longrightarrow V_2$ is a $R_\ell$-linear isomorphism  $\infty$-tangent to the identity then $(1+\Delta_\gamma)$ is a nuclear endomorphism of $V_1[\![Z]\!]$. 
\end{Lem}
\begin{proof}
Let $N\in \N^*$ verify the previous condition and $i$ as in Definition \ref{tangence}. Let $m<N$. As in \cite[Lemma 5.1.5]{ref4} we have the following equality
$$\delta_m  =(\delta_E(\theta)-\gamma_i^{-1}\delta_E(\theta)\gamma_i)\delta_E(\theta)^{m-1}=\gamma_i^{-1}(\gamma_i-1)\delta_E(\theta)^m+\gamma_i^{-1}\delta_E(\theta)(1-\gamma_i )\delta_E(\theta)^{m-1}.$$

We can show that for $j\geq \max(\lfloor q^n\dfrac{m}{q^n}\rfloor +c+i, \lfloor q^n\dfrac{m-1}{q^n}\rfloor +c+i)$, for $x\in(\theta^{-j}\mathcal{O}_{L_\infty,\ell})^n$,  then $\delta_E^m(\theta)x$ and $\delta_E^{m-1}(\theta)x$ belong to $ \theta^{-i}\mathcal{O}_{L_\infty,\ell}^n$.

As $\gamma$ is $N'$-tangent to the identity, for $x\in \theta^{-j}\mathcal{O}_{L_\infty,\ell}^n$ we have the inequalities : 
\begin{equation*}
\begin{aligned} 
\vert\vert\delta_m^i(x) \vert\vert&\leq \max \{\vert\vert \theta\vert\vert^{-N'-a}\vert\vert(\delta_E(\theta)^mx)\vert\vert;\vert\vert \delta_E(\theta)\vert\vert\vert\vert \delta_E(\theta)^{m-1}x-\gamma_i \delta_E(\theta)^{m-1}x\vert\vert\} \\
&  \leq \max \{\vert\vert \theta\vert\vert^{-N'-a}\vert\vert \theta\vert\vert^{q^n\lfloor m/q^n\rfloor+c}   \vert\vert x\vert\vert; \vert\vert \theta\vert\vert^{c}q^{-N'-a}\vert\vert \theta\vert\vert^{q^n \lfloor (m-1)/q^n\rfloor+c}\vert\vert x\vert\vert\} \\
&  \leq q^{-1-a}\vert\vert x\vert\vert
\end{aligned}
\end{equation*} 
As $U_{j,\infty}\subset \theta^{-j}\mathcal{O}_{L_\infty,\ell}$ and $\theta^{-1-a-j}\mathcal{O}_{L_\infty,\ell}\subset U_{j+1,\infty}$, $\delta_m$ is locally contracting.

\end{proof}

\subsection{Endomorphisms of $L_{\infty,\ell}^n/\Lambda$}\label{53} ${}$\par
We look at the particular case where $V_1=V_2=(L_{\infty,\ell})^n/\Lambda$  with $\Lambda$ a projective  $\ell[\theta][G]$-lattice. We fix $r$ so that $(\theta^{-r}\mathcal{O}_{L_\infty,\ell})^n\cap \Lambda=\{0\}$ and $a$ such that for  $j\geq r$, $$\theta^{-a-j}\mathcal{O}_{L_\infty,\ell}\subset \mathcal{U}_{j,\infty}\subset \theta^{-j}\mathcal{O}_{L_\infty,\ell}.$$

\begin{definition}
A $R_\ell$-linear continuous endomorphism  $\phi:L_{\infty,\ell}^n/\Lambda\longrightarrow L_{\infty,\ell}^n/\Lambda$ is called a {\it local $m$-contraction} for $m\in\mathbb{N}$ if there exists $i\geq r$ such that $$\vert\vert\phi(x)\vert\vert\leq\vert\vert \theta\vert\vert^{-m}\vert\vert x \vert\vert \hspace{0,1cm} \text{for all} \hspace{0,1cm} x\in (\theta^{-i}\mathcal{O}_{L_\infty,\ell})^n.$$ 
\end{definition}

\begin{rem}\label{remarqueloc}
If $\phi$ is a local $m$-contraction for $m>a$ then $\phi$ is locally contracting on $(L_{\infty,\ell})^n/\Lambda$ for $\mathcal{U}^n$. So  $\det(1+\phi\vert (L_{\infty,\ell})^n/\Lambda)$ is well defined.  Indeed, let $i\geq r$ such that $\vert\vert\phi(x)\vert\vert\leq\vert\vert \theta\vert\vert^{-m}\vert\vert x \vert\vert $ for all $x\in \theta^{-i}\mathcal{O}_{L_\infty,\ell}$. For $j\geq i$, we have 

$$\phi(U_{j,\infty}^n)\subset \phi(\theta^{-j}\mathcal{O}_{L_\infty,\ell}^n)\subset \theta^{-j-m}\mathcal{O}_{L_\infty,\ell}^n\subset \theta^{-j-a-1}\mathcal{O}_{L_\infty,\ell}^n\subset U_{j+1,\infty}^n.$$ So $\phi$ is locally contracting for  $\mathcal{U}^n$ and $U_{i,\infty}^n$ is a  nucleus. 
\end{rem}

We recall that $\vert\vert \delta_E(\theta)\vert\vert=q^d$. 
\begin{proposition}\label{contraction}
Let  $\gamma : (L_{\infty,\ell})^n/\Lambda\longrightarrow (L_{\infty,\ell})^n/\Lambda$ be a  $R_\ell$-linear continuous isomorphism which is  $N$-tangent to the identity for $N>0$. Let $\psi:(L_{\infty,\ell})^n/\Lambda\longrightarrow (L_{\infty,\ell})^n/\Lambda$ a local $m$-contraction, $R_\ell$-linear for $m>2a+d$. Let $f=\delta_E(\theta)\gamma$. 
Then \begin{itemize}
\item $f\psi$ and $\psi f$ are local $(m-d)$-contractions of $(L_{\infty,\ell})^n/\Lambda$.
\item $\det(1+f\psi\vert (L_{\infty,\ell})^n/\Lambda)=\det(1+\psi f\vert (L_{\infty,\ell})^n/\Lambda)$.
\end{itemize}
\end{proposition}

\begin{proof}
Let $i>r$ be such that $\gamma:\theta^{-i}\mathcal{O}_{L_\infty,\ell}^n\longrightarrow \theta^{-i}\mathcal{O}_{L_\infty,\ell}^n$ is a bijective isometry and  $\vert\vert \psi(x)\vert\vert\leq \vert\vert \theta\vert\vert^{-M}\vert\vert x \vert\vert $ for all $x\in \theta^{-(i-c)}\mathcal{O}_{L_\infty,\ell}^n$. 

For the first assertion, for $x\in \theta^{-i}\mathcal{O}_{L_\infty,\ell}^n$, as $\vert\vert \psi f(x)\vert\vert= \vert\vert\psi (\delta_E(t)\gamma(x))\vert\vert$, we get : 
 $$\vert\vert \psi f(x)\vert\vert\leq\vert\vert \theta\vert\vert^{-m} \vert\vert \delta_E(t)\gamma(x)\vert\vert\leq \vert\vert \theta\vert\vert^{-m} \vert\vert \delta_E(t)\vert\vert\vert\vert x \vert\vert\leq \vert\vert \theta\vert\vert^{-m+d}\vert\vert x \vert\vert.$$  

In the same way we get $\vert\vert \psi f(x)\vert\vert \leq  \vert\vert \theta\vert\vert^{-m+d} \vert\vert x \vert\vert$ for all $x\in \theta^{-i}\mathcal{O}_{L_\infty,\ell}^n$. Thus $f\psi$ and $\psi f$ are $M-d$-contractions over $L_{\infty,\ell}^n/\Lambda$ and then locally contracting by choice of  $m$ by \ref{remarqueloc}. 

For the second assertion, we proceed as \cite[Proposition 5.2.3]{ref4}. As   $\gamma$ is an isomorphism and $L_{\infty,\ell}^n/\Lambda$  is divisible (because $\delta_E(\theta)$ is invertible) then $f$ is surjective. We have $f^{-1}(U_{i,\infty}^n)=\gamma^{-1}(\delta_E(\theta)^{-1}\Lambda/\Lambda)\oplus\gamma^{-1}(\delta_E(\theta)^{-1}U_{i_\infty}^n)$. We set $f^{-1}(U_{i,\infty}^n)^*=\gamma^{-1}(\delta_E(\theta)^{-1}U_{i_\infty}^n)$. We take $\delta_E(\theta)^{-i}U_{i_\infty}^n$ instead of $U_{i+1,\infty}$. We can show that it is a common nucleus for $f\psi$ and $\psi f$.

\end{proof}

We recall that $\max_{i\in [\![0;q^n-1]\!]}\vert\vert\delta_E(\theta)^i\vert\vert=q^c$.
\begin{rem}
Let $\gamma : (L_{\infty,\ell}/\Lambda)^n\longrightarrow (L_{\infty,\ell}/\Lambda)^n$ a $R_\ell$-linear continuous isomorphism which is $N$-tangent to the identity $N>0$. Let $\psi:(L_{\infty,\ell}/\Lambda)^n\longrightarrow (L_{\infty,\ell}/\Lambda)^n$ be a local  $m$-contraction,  $R_\ell$-linear. Let $f=\delta_E(\theta)\gamma$. 

Let $\phi\in R_\ell\{f,\psi\}$ be a sum of  monomials of degree at most $r<m$ such that each monome contain a $\psi$. Then we can show  that $\phi$ is a  $m-r-c+1$ contraction over $L_{\infty,\ell}^n/\Lambda$.
 \end{rem}

\begin{Cor}\label{tangent1}
Let $\gamma : (L_{\infty,\ell})^n/\Lambda\longrightarrow (L_{\infty,\ell})^n/\Lambda$ be a $R_\ell$-linear continuous isomorphism which is  $2N$-tangent to the identity with $N>a+c+d-1$. Then $$\det_{R_\ell[\![Z]\!]/Z^N}\left(1+\Delta_\gamma\right\vert V[\![Z]\!]/Z^N)=1.$$

\end{Cor}
\begin{proof}
It is the same proof as \cite[Corollary 5.2.9]{ref4} using the new conditions on $N$. 
\end{proof}

\subsection{Relation volume-determinant} ${}$\par
The goal of this section is to obtain some relations between volume and determinants. It will be used later because of the fact that the Fitting ideals can be expressed as determinants. 
\bigskip

\begin{Lem}\label{independance}
Let $V_1, V_2=(L_{\infty,\ell}/\Lambda_2)^n$ where $\Lambda_2$ is a projective $\ell[\theta][G]$-lattice of $L_{\infty,\ell}$. They belong to $\mathcal{C}$. Let $\gamma_1,\gamma_2:V_1\cong V_2$ be two isomorphisms $R_\ell$-linear, continuous and $2N'$-tangent to the identity. Let $N>a+c+d-1$ verify $N'-q^n\lfloor\dfrac{m}{q^n}\rfloor-c\geq 1 \text{ and } N'-q^n\lfloor\dfrac{m-1}{q^n}\rfloor-2c\geq 1$. Then 

$$\det_{R_\ell[\![Z]\!]/Z^N}(1+\Delta_{\gamma_1}\vert V_1)=\det_{R_\ell[\![Z]\!]/Z^N}(1+\Delta_{\gamma_2}\vert V_1).$$ 
\end{Lem}
\begin{proof}
It is the same proof as \cite[Lemma 5.3.4]{ref4} using the new conditions on $N$. 
\end{proof}

We recall that $\Lambda_0$ is a fixed projective $\ell[\theta][G]$-lattice used to define the volume and that the quotient of volumes does not depend on this choice.

\begin{Lem} Let $V_1\in \mathcal{C}$ and $\Lambda_1$ be the associated lattice and $V_2=(L_{\infty,\ell}/\Lambda_2)^n$. We assume there exists $\Lambda$  a $\ell[\theta][G]$-lattice which contains $\Lambda_1$ and $\Lambda_2^n$. Let $\gamma:V_1\cong V_2$ a $R_\ell$-linear isomorphism  $2N'$-tangent to the identity. Let $N>a+c+d-1$ verify $N'-q^n\lfloor\dfrac{m}{q^n}\rfloor-c\geq 1 \text{ and } N'-q^n\lfloor\dfrac{m-1}{q^n}\rfloor-2c\geq 1$. Then $$\det_{R_\ell[\![Z]\!]/Z^N}(1+\Delta_{\gamma}\vert V_1)\equiv \dfrac{\Vol_{\Lambda_0}(V_2)}{\Vol_{\Lambda_0}(V_1)}\mod \theta^{-N}.$$
\end{Lem}

\begin{proof}
It is the same proof as \cite[Lemma 5.3.6]{ref4} using the new conditions on $N$. 
\end{proof}
These two lemmas permit us to obtain the following theorem. 

\begin{theorem}\label{relation}
Let $V_1, V_2=(L_{\infty,\ell}/\Lambda_2)^n$ where $\Lambda_2$ is a projective $\ell[\theta][G]$-lattice be two modules in $\mathcal{C}$ and $\gamma:V_1\cong V_2$ a $R_\ell$-linear, continuous isomorphism $\infty$-tangent to the identity. Then $$\displaystyle\det_{R_\ell[\![Z]\!]}(1+\Delta_{\gamma}\vert V_1) =\dfrac{\Vol_{\Lambda_0}(V_2)}{\Vol_{\Lambda_0}(V_1)}.$$ 
\end{theorem}

\begin{proof}
It is almost the same proof as \cite[Theorem 5.3.2]{ref4}. We fix $N'$ and we take $\phi$ a morphism which is $2N'$-tangent to the identity as in the proof of \cite[Theorem 5.3.2]{ref4}. For $N$ verifying some conditions we obtain the same equalities modulo $\theta^{-N}$. If we take $N'\rightarrow \infty$, $N$ can be as big as we want and we get the desired result. 
\end{proof}

\section{Equivariant class formula à la Taelman}
\subsection{An equivariant class formula over $\mathbb{F}_q(z)$ for $t$-module } ${}$\par

The goal of this section is to prove the equivariant class formula over $\F_q(z)$ for $t$-modules. 

Let    $L/k$ be  a finite extension and $E$ a $t$-module defined over $\mathcal{O}_K$ where $K$ verifies $k\subset K\subset L$ and $L/K$ is a finite abelian extension of Galois group $G$.  Let $M$ be an almost taming module for $L/K$.

We now take $\ell=\F_q(z)$. It means that $L_{\infty,\ell}=\widetilde{L}_\infty$, $M_\ell=\widetilde{M}$. 

We recall that $\mathcal{C}$ corresponds to compact $\F_q(z)[\theta][G]$-modules  $V$ which are $G$-cohomologically trivial and verify an exact sequence of $\widetilde{A}[G]$-modules 
$$0\longrightarrow \Lie_{\widetilde{E}}(\widetilde{L}_\infty)/\Lambda\longrightarrow V\longrightarrow H\longrightarrow 0$$ where $\Lambda$ is an  $\widetilde{A}[G]$-lattice of $\Lie_{\widetilde{E}}(\widetilde{L}_\infty)$ and $H$ is an  $\widetilde{A}[G]$-module and a $\F_q(z)$-vector space of finite dimension. 

\begin{proposition}\label{uniteproj}
$U(\widetilde{E}(\widetilde{M}))$ is a projective $\widetilde{A}[G]$-module. 

\end{proposition}
\begin{proof}
We have the exact sequence $$0\longrightarrow\Lie_{\widetilde{E}}(\widetilde{L}_\infty)/U(\widetilde{E}(\widetilde{M}))\longrightarrow \widetilde{E}(\widetilde{L}_\infty)/\widetilde{E}(\widetilde{M})\longrightarrow H(\widetilde{E}(\widetilde{M}))\longrightarrow0.$$

As $H(\widetilde{E}(\widetilde{M}))=\{0\}$, we have the isomorphism of  $\widetilde{A}[G]-$modules  $\widetilde{E}(\widetilde{L}_\infty)/\widetilde{E}(\widetilde{M})\cong \Lie_{\widetilde{E}}(\widetilde{L}_\infty)/U(\widetilde{E}(\widetilde{M}))$ induced by $\exp_{\widetilde{E}}$. 
 $\widetilde{E}(\widetilde{L}_\infty)/\widetilde{E}(\widetilde{M})$ is $G$-cohomologically trivial because $L_\infty$ is (normal basis theorem) and also $\widetilde{M}$  (because $\widetilde{M}$ is a projective $\widetilde{A}[G]$-module). It implies that $\Lie_{\widetilde{E}}(\widetilde{L}_\infty)/U(\widetilde{E}(\widetilde{M}))$ is $G$-cohomologically trivial. Furthermore, we have the exact sequence of $\widetilde{A}[G]$-modules, $$0\longrightarrow  U(\widetilde{E}(\widetilde{M}))\longrightarrow \Lie_{\widetilde{E}}(\widetilde{L}_\infty) \longrightarrow \Lie_{\widetilde{E}}(\widetilde{L}_\infty)/U(\widetilde{E}(\widetilde{M}))\longrightarrow 0.$$
It implies that $U(\widetilde{E}(\widetilde{M}))$ is $G$-cohomologically trivial. As $\widetilde{A}$ is a Dedekind ring and $U(\widetilde{E}(\widetilde{M}))$ is a projective $\widetilde{A}$-module (because it is a free $\widetilde{A}$-module), by Theorem \ref{Projct}, $U(\widetilde{E}(\widetilde{M}))$ is a projective $\widetilde{A}[G]$-module. 
\end{proof}

For the next theorem, we use the same line of proof as \cite[Thm. 6.1.1]{ref4} but we recall it for the convenience of the reader.
\begin{theorem}
We have the equality in $1+\theta^{-1}\mathbb{F}_q(z)[\![\theta^{-1}]\!][G]$ : 

$$ \mathcal{L}_G(\widetilde{E}( \widetilde{M}))=\dfrac{\Vol_{\Lambda_0}(\widetilde{E}(\widetilde{L}_\infty)/\widetilde{E}(\widetilde{M}))}{\Vol_{\Lambda_0}(\Lie_{\widetilde{E}}(\widetilde{L}_\infty)/\Lie_{\widetilde{E}}(\widetilde{M}))}.$$ 
\end{theorem}

\begin{proof}
We set $V_1=\widetilde{E}(\widetilde{L}_\infty)/\widetilde{E}(\widetilde{M})$ and $V_2=\Lie_{\widetilde{E}}(\widetilde{L}_\infty)/\Lie_{\widetilde{E}}(\widetilde{M})$. We have the isomorphism of $\widetilde{A}[G]$-modules induced by $\exp_{\widetilde{E}}$ because of the triviality of $H(\widetilde{E}(\widetilde{M}))$ $$ \dfrac{\Lie_{\widetilde{E}}(\widetilde{L}_\infty)}{\exp_{\widetilde{E}}^{-1}(\widetilde{M})} \cong\dfrac{\widetilde{E}(\widetilde{L}_\infty)}{\widetilde{E}(\widetilde{M})}.$$ 

By Proposition \ref{reseautildea},   $\exp_{\widetilde{E}}^{-1}(\widetilde{M})$ is an $\widetilde{A}[G]$-lattice. So $V_1$ et $V_2$ belong to $\mathcal{C}$ (both $H$ are trivial). 

They have the same $\mathbb{F}_q(z)[G]$-structure but $\theta$ does not act on them in the same way. $\theta$ acts on $V_2$ via $\delta_E$ but via $\phi_{\widetilde{E}}$ on $V_1$. We set $\gamma=\id: V_1\longrightarrow V_2$ which is $\mathbb{F}_q(z)[G]$-linear.  

The map $\exp_{\widetilde{E}} :\Lie_{\widetilde{E}}(\widetilde{L}_\infty)\longrightarrow \widetilde{E}(\widetilde{L}_\infty)$ is convergent everywhere. We denote by $\overline{\exp_{\widetilde{E}}}:V_2\longrightarrow V_1$ the induced application by $\exp_{\widetilde{E}}$. 
As $\overline{\exp_{\widetilde{E}}}=\gamma\circ \iota_1$, by Proposition \ref{infiniment}, $\gamma$ is $\infty$-tangent to the identity. Thus, we can apply Theorem \ref{relation} to obtain 

 $$\det_{\F_q(z)[G][\![Z]\!]}(1+\Delta_{\gamma})\vert \widetilde{E}(\widetilde{L}_\infty)/\widetilde{E}(\widetilde{M})) =\dfrac{\Vol_{\Lambda_0}(\Lie_{\widetilde{E}}(\widetilde{L}_\infty)/\Lie_{\widetilde{E}}(\widetilde{M}))}{\Vol_{\Lambda_0}(\widetilde{E}(\widetilde{L}_\infty)/\widetilde{E}(\widetilde{M}))}.$$

Furthermore, as $\gamma=\id$, we can show that $1+\Delta_{\gamma}=\dfrac{1-\delta_E(\theta).Z}{1-\phi_{\widetilde{E}}(\theta) Z}.$

By Corollary \ref{fonctionl2}, we have : $$\det_{\F_q(z)[G][\![Z]\!]}\left(\dfrac{1-\delta_E(\theta).Z}{1-\phi_{\widetilde{E}}(\theta) Z}\right)= \mathcal{L}_G(\widetilde{E}( \widetilde{M}))^{-1}.$$

It implies $$\mathcal{L}_G(\widetilde{E}( \widetilde{M}))=\dfrac{\Vol_{\Lambda_0}(\widetilde{E}(\widetilde{L}_\infty)/\widetilde{E}(\widetilde{M}))}{\Vol_{\Lambda_0}(\Lie_{\widetilde{E}}(\widetilde{L}_\infty)/\Lie_{\widetilde{E}}(\widetilde{M}))}.$$

\end{proof}
\bigskip
It allows us to have the following theorem. 
\begin{theorem}\label{Formuleaaa}

We have
$$\left[\Lie_{\widetilde{E}}(\widetilde{M}):U(\widetilde{E}(\widetilde{M}))\right]_G=\mathcal{L}_G(\widetilde{E}( \widetilde{M})).$$
\end{theorem}

\begin{proof}
As $H(\widetilde{E}(\widetilde{M}))=\{0\}$, we have the exact sequence $$0\longrightarrow \dfrac{\Lie_{\widetilde{E}}(\widetilde{L}_\infty)}{\exp_{\widetilde{E}}^{-1}(\widetilde{M})} \longrightarrow \dfrac{\widetilde{E}(\widetilde{L}_\infty)}{\widetilde{E}(\widetilde{M})}\longrightarrow0.$$ 

By the previous theorem,  we obtain \begin{equation*}
\begin{aligned} 
\mathcal{L}_G(\widetilde{E}( \widetilde{M})) &=\dfrac{\Vol_{\Lambda_0}(\widetilde{E}(\widetilde{L}_\infty)/\widetilde{E}(\widetilde{M}))}{\Vol_{\Lambda_0}(\widetilde{L}_\infty/\widetilde{M})}\\
& =\dfrac{1}{\left[U(\widetilde{E}(\widetilde{M})):\Lambda_0\right]_G}\left[\Lie_{\widetilde{E}}(\widetilde{M}):\Lambda_0\right]_G\\
& =\left[\Lie_{\widetilde{E}}(\widetilde{M}):U(\widetilde{E}(\widetilde{M}))\right]_G.
\end{aligned}
\end{equation*} 

\end{proof}

\subsection{Determinant} ${}$\par
We see in this section how we get to the equivariant to the classical setting. 
\bigskip

Let $G$ be a finite abelian group. Let $x\in k_\infty[G]$, we denote by  $M_x$ the matrix of the multiplication by $x$ in $ k_\infty[G]$ seen as a $k_\infty$ endomorphism of $ k_\infty[G]$.
We define $\displaystyle\det_G$ as for $x\in k_\infty[G]$,  $\displaystyle\det_G(x)=\det(M_x)$.  

For $i,j=1, \ldots, \vert G \vert$, we denote by $[i;j]$ the integer $r\in [\![1; \vert G \vert ]\!]$ such that $g_rg_j=g_i$.

In particular, if $x=\sum\limits_{i=1}^{\vert G \vert}x_ig_i $, we have for $i,j=1, \ldots, \vert G \vert$, $(M_x)_{(i;j)}=x_{[i;j]}$.

\begin{proposition}
We have $x$ invertible in $A[G]$ if and only if $\displaystyle\det_Gx\in \F_q^*$. 
\end{proposition}
\begin{proof}
First, we suppose that $x$ is invertible in  $A[G]$. So there exists $y\in A[G]^*$ such that $xy=\id$. So we have $M_x, M_y \in \mathcal{M}_{\vert G\vert }(A)$ such that $M_xM_y=M_{\id}=I_{\vert G \vert }$. So $\displaystyle\det_G(x)\displaystyle\det_G(y)=1$. It implies that $\displaystyle\det_G(x)\in A^*=\F_q^*$.

Reciprocally, we suppose $\displaystyle\det_G(x)\in \F_q^*$.

We have $M_x^{-1}=\dfrac{1}{\displaystyle\det_G(x)} ^tCom(M_x)\in \mathcal{M}_{\vert G\vert }(A)$.
For $i=1,\ldots, \vert G\vert$, we let $y_i=(M_x^{-1})_{(i;1)}$ and $y=\sum\limits_{i=1}^{\vert G \vert}y_ig_i$.
As for $i,j=1, \ldots, \vert G \vert$, $[i;j]$ is the integer  $r\in [\![1; \vert G \vert ]\!]$ such that $g_rg_j=g_i$, we have  $$xy=\sum\limits_{r=1}^{\vert G \vert}\sum\limits_{j=1}^{\vert G \vert}x_ry_jg_kg_j=\sum\limits_{i=1}^{\vert G \vert}\left(\sum\limits_{j=1}^{\vert G \vert}y_jx_{[i;j]}\right)g_i.$$ 
As $\sum\limits_{j=1}^{\vert G \vert}y_jx_{[i;j]}$ correspond to the coefficient lign $i$ column 1 of the matrix $M_xM_x^{-1}$, we have $xy=\id$. So $x$ is invertible in $A[G]$. 
\end{proof}

We have the theorem of  Kovacs, Silver and Williams \cite{ref7} that we will use to prove the next propositions.  
\begin{theorem}\label{detttt}
Let $R$ be a commutative ring. Assume that $N$ is a block matrix $m\times m$ of  blocks $N^{i,j}\in \mathcal{M}_n(R)$ that commute pairwise. Then $$\vert N\vert =\left\vert \sum\limits_{\sigma \in S_m}\varepsilon(\sigma)N^{1,\sigma(1)}N^{2,\sigma(2)}\ldots N^{m,\sigma(m)}\right\vert.$$
\end{theorem}

\begin{proposition}
Let $P$ and $Q$ be  two free $A[G]$-lattices of $L_\infty^n$. We have 
$$\displaystyle\det_G[P:Q]_{A[G]}=[P:Q]_{A}.$$
\end{proposition}
\begin{proof}

Let $P$ and $Q$ be two free $A[G]$-lattices of rank $m$, with respective bases $(e_1, \ldots e_{m})$ and $(f_1,\ldots, f_{m})$. There exists $X\in Gl_m(k_\infty[G])$ which sends  $(f_1,\ldots, f_m) $ over $(e_1,\ldots, e_m)$. We denote $(X)_{i;j}=x_{i,j}$. We have $$[P:Q]_{A[G]}=\det X= \sum\limits_{\sigma \in S_m}\varepsilon(\sigma)x_{1,\sigma(1)}x_{2,\sigma(2)}\ldots x_{m,\sigma(m)}.$$ 
Thus $\displaystyle\det_G[P:Q]_{A[G]}=\left\vert \sum\limits_{\sigma \in S_m}\varepsilon(\sigma)M_{x_{1,\sigma(1)}}M_{x_{2,\sigma(2)}}\ldots M_{x_{m,\sigma(m)}}\right\vert.$
We denote by  $(e_1',\ldots, e_{m\vert G\vert }')$ and $(f_1',\ldots, f_{m\vert G\vert }')$ the bases of $P$ and $Q$ as $A$-modules. The matrix which sends $(f_1',\ldots, f_{m\vert G\vert }')$ over $(e_1',\ldots, e_{m\vert G\vert }')$ is a block matrix $m\times m$ where the blocks are $M_{x_{i,j}}\in \mathcal{M}_{\vert G\vert}(k_\infty)$. They commute pairwise as they represent multiplications by $x_{i,j}$. 
As $[P:Q]_{A}=\det M$, by the previous theorem, we have $$\displaystyle\det_G[P:Q]_{A[G]}=\left\vert \sum\limits_{\sigma \in S_m}\varepsilon(\sigma)M_{x_{1,\sigma(1)}}M_{x_{2,\sigma(2)}}\ldots M_{x_{m,\sigma(m)}}\right\vert=\det M=[P:Q]_{A}.$$
\end{proof}

It allows us to obtain the following proposition. 
\begin{proposition}\label{detgreseau}
Let $P$ and $Q$ be two projective $A[G]$-lattices of $L_\infty^n$. We have 
$$\displaystyle\det_G[P:Q]_{A[G]}=[P:Q]_{A}.$$
\end{proposition}
\begin{proof}
Let $F$ and $F'$ be two free $A[G]$-lattices such that $P\subset F$ et $Q\subset F'$. We recall that by definition  
$[P:Q]_{A[G]}=[F:F']_{A[G]}\dfrac{\vert F'/Q\vert_{A[G]}}{\vert F/P\vert_{A[G]}}.$ As $\det G$ is multiplicative, we get $\displaystyle\det_G[P:Q]_{A[G]}= \displaystyle\det_G[F:F']_{A[G]}\dfrac{\displaystyle\det_G\vert F'/Q\vert_{A[G]}}{\displaystyle\det_G\vert F/P\vert_{A[G]}}.$ By the previous propositions, we have $$\displaystyle\det_G[P:Q]_{A[G]}=[F:F']_{A}\dfrac{\vert F'/Q\vert_{A}}{\vert F/P\vert_{A}}=[P:Q]_{A}.$$
\end{proof}
We have a similar result with the unitary generator of the Fitting ideals. 
\begin{proposition}\label{detgfitting}
Let $N$ be  $A[G]$-module which is finitely generated anf projective as a $\F_q[G]$-module. Then $\displaystyle\det_G\vert N\vert_G=\vert N\vert_A$. 
\end{proposition}
\begin{proof}
First we suppose that $\F_q[G]$ is local. We denote $m=rank(N)$. As $N$ is projective and $\F_q[G]$ is local, $N$ is a free $\F_q[G]$-module. We have 

\begin{equation*}
\begin{aligned} 
\vert N\vert_G &=\det_{A[G]}(\theta I_m -A_\theta)\\
& =\sum\limits_{\sigma \in S_m}\varepsilon(\sigma)(\theta I_m -A_\theta)_{1,\sigma(1)}(\theta I_m -A_\theta)_{2,\sigma(2)}\ldots (\theta I_m -A_\theta)^{m,\sigma(m)}
\end{aligned}
\end{equation*} 
 where $A_\theta$ is the multiplication by  $\theta$ in a $\F_q[G]$-basis of $N$. 
So we obtain 
\begin{equation*}
\begin{aligned} 
\displaystyle\det_G\vert N\vert_G &=\displaystyle\det_G\left(\sum\limits_{\sigma \in S_m}\varepsilon(\sigma)\prod\limits_{i=1}^{\vert G\vert}(\theta I_m -A_\theta)_{i,\sigma(i)}\right)\\
& =\det \left(\sum\limits_{\sigma \in S_m}\varepsilon(\sigma)\prod\limits_{i=1}^{\vert G\vert}M_{(\theta I_m -A_\theta)_{i,\sigma(i)}}\right).
\end{aligned}
\end{equation*} 

 We denote by $B_\theta$ the multiplication by $\theta$ in a $\F_q$-basis of $N$. So we have  that $\theta I_{\vert G\vert m}-B_\theta$ is a block matrix where the blocks are $M_{(\theta I_m -A_\theta)_{i,\sigma(i)}}$.
  As they commute pairwise, by  Theorem \ref{detttt} $$\vert N\vert_A=\det(\theta I_{\vert G\vert m}-B_\theta)\left\vert \sum\limits_{\sigma \in S_m}\varepsilon(\sigma)\prod\limits_{i=1}^{\vert G\vert}M_{(\theta I_m -A_\theta)_{i,\sigma(i)}}\right\vert.$$ So we have  $\displaystyle\det_G\vert N\vert_G=\vert N\vert_A$. 

\end{proof}

\begin{rems}\label{detg}
By Proposition \ref{detgfitting}, we obtain $\displaystyle\det_G(\mathcal{L}(\phi(M)))=\mathcal{L}(\phi(M))$. For $N$ and $N'$ two projective $A[G]$-lattices, by Proposition \ref{detgreseau}, we have $\displaystyle\det_G([N:N']_{A[G]})=[N:N']_{A}$. 

\end{rems}

\subsection{A partial equivariant class formula for $t$-modules} ${}$\par

In this last section, we will see a partial equivariant class formula over $\F_q$ for $t$-modules and some sufficient conditions to obtain this equivariant class formula. 

By Ferrara, Green, Higgins and D. Popescu \cite{ref4} (\ref{fonctionl2} when $\ell=\F_q$), we have the convergence of $\mathcal{L}_G(E( M))$. 

Following Anglès, Ngo Dac, Tavares Ribeiro \cite{ref1}, we obtain the next theorem. 
\begin{theorem}\label{sousmod}
Let $\Lambda$ be a projective $A[G]$-module such that $\Lambda\subset U_{St}(E(M))$ and $\Lambda$ is a $A$-lattice of $L_\infty$. 
 Let $E$ be a $t$-module defined over $\mathcal{O}_K$. Then
$$\mathcal{L}_G(E( M))^{-1}[ \Lie_E(M):\Lambda]_{A[G]}\in A[G].$$
Furthermore, 
$$\displaystyle\det_G\left(\dfrac{[ \Lie_E(M):\Lambda]_{A[G]}}{\mathcal{L}_G(E( M))}\right)=[ U_{St}(E(M)):\Lambda]_{A}.$$ 
\end{theorem}
\begin{proof}
Let $P$ be a monic prime of $A$. By Proposition \ref{fitt}, there exists  $x_P\in A[G][z]$ such that \begin{itemize}
\item $\Fitt_{\widetilde{A}[G]}\left(\widetilde{E}\left(\dfrac{\widetilde{M}}{P\widetilde{M}}\right)\right)=x_P\widetilde{A}[G]$,
 \item $\Fitt_{A[G]}\left(\dfrac{ M}{PM}\right)=x_P(0)A[G]$,
 \item $\Fitt_{A[G]}\left(E\left(\dfrac{ M}{PM}\right)\right)=x_P(1)\widetilde{A}[G]$. 
\end{itemize}

We obtain that $\prod\limits_{P\not\in S}\dfrac{x_P(0)}{x_P}$ converges in $\widetilde{k}_\infty [G]$  by Corollary \ref{fonctionl2} when $\ell=\F_q(z)$.

Let $\mathcal{B}$ be a basis of $\Lie_E(L_\infty)$ over $k_\infty[G]$.  It is also a basis of $\Lie_{\widetilde{E}}(\widetilde{L}_\infty)$ over $\widetilde{k}_\infty[G]$ and a basis of $\Lie_{\widetilde{E}}(\mathbb{T}_z(L_\infty))$ over $\mathbb{T}_z(k_\infty)[G]$.  We denote by $N$ the $A[G]$-module spanned by $\mathcal{B}$.
 We have $\left[\Lie_{\widetilde{E}}(\widetilde{M}):\widetilde{N}\right]_{\widetilde{A}[G]}=\F_q(z)[\Lie_E(M):N]_{A[G]}$ and $$\left[\Lie_{\widetilde{E}}(\widetilde{M}):U(\widetilde{E}(\widetilde{ M}))\right]_{\widetilde{A}[G]}=\left[\Lie_{\widetilde{E}}(\widetilde{M}):\widetilde{N}\right]_{\widetilde{A}[G]}\left[\widetilde{N}:U(\widetilde{E}(\widetilde{ M}))\right]_{\widetilde{A}[G]}.$$

By Theorem \ref{Formuleaaa}, $$\left[\Lie_{\widetilde{E}}(\widetilde{ M}):U(\widetilde{E}( \widetilde{M}))\right]_{\widetilde{A}[G]}\subset \mathcal{L}_G(\widetilde{E}(\widetilde{ M})).$$

Furthermore, by Proposition \ref{engendre}, 
 $U(\widetilde{E}(\widetilde{ M}))$ is the $\F_q(z)$-vector space generated by $U(\widetilde{E}( M[z]))$. 

As $\Lambda \subset U_{St}(E( M)))$, the determinant of elements of  $\Lambda$ in the basis $\mathcal{B}$ comes  from the evaluation at $z=1$ of $U(\widetilde{E}( M[z]))$. 

As $\left[\widetilde{N}:U(\widetilde{E}( \widetilde{M}))\right]_{\widetilde{A}[G]}\subset \left[\Lie_{\widetilde{E}}(\widetilde{M}):\widetilde{N}\right]_{\widetilde{A}[G]}^{-1} \mathcal{L}_G(\widetilde{E}(\widetilde{ M}))$, we have $$[N:\Lambda]_{A[G]}\subset [M:N]_{A[G]}^{-1} \mathcal{L}_G(E( M)).$$

It implies that $[N:\Lambda]_{A[G]}[\Lie_E(M):N]_{A[G]}\subset  \mathcal{L}_G(E( M))$.

Thus we have $$[ \Lie_E(M):\Lambda]_{A[G]}\subset  \mathcal{L}_G(E( M)).$$

Moreover, we have $[ \Lie_E(M):U_{St}(E( M))]_{A}= \mathcal{L}(E( M)).$

By Remarks \ref{detg}, it implies that
$$\displaystyle\det_G\left(\dfrac{[ \Lie_E(M):\Lambda]_{A[G]}}{\mathcal{L}_G(E( M))}\right)=\dfrac{[ \Lie_E(M):\Lambda]_{A}}{\mathcal{L}(E( M))}=[ U_{St}(E(M)):\Lambda]_{A}.$$ 

\end{proof}

\begin{rem} In particular, if $U_{St}(E(M))$ is a projective $A[G]$-module, then $$[ \Lie_E(M):U_{St}(E( M))]_{A[G]}= \mathcal{L}_G(E( M)).$$
\end{rem}
\bigskip

Now we show some sufficient conditions for $U_{St}(E(M))$ to be a projective $A[G]$-module.

We recall that for $f\in A[z]$, $H(\widetilde{E}(M[z]))[f]=\{x\in H(\widetilde{E}(M[z])), fx=0\}.$
\hyphenation{co-ho-mo-lo-gi-que-ment}
\begin{proposition}\label{proj}
\sloppy Let $E$ be a $t$-module such that $H(\phi(M))$ is $G$-cohomologically  trivial. Then $U(E(M))$ and $U_{St}(E(M))$ are projective $A[G]$-modules. Moreover,  $$[ \Lie_E(M):U_{St}(E(M))]_{A[G]}=\mathcal{L}_G(E( M)).$$ 

\end{proposition}

\begin{proof}
We have the exact sequence $$0\longrightarrow\dfrac{\Lie_E(L_\infty)}{U(E(M))}\longrightarrow \dfrac{E(L_\infty)}{E(M)}\longrightarrow H(E(M)) \longrightarrow 0.$$

As $M$ is a projective $A[G]$-module, it is $G$-cohomologically trivial. As $L_\infty$ and $H(E(M))$ are also $G$-cohomologically trivial, it is the same for $U(E(M)).$  As $A$ is a Dedekind ring and that $U(E(M))$ is a projective $A$-module, it is a projective $A[G]$-module by Theorem \ref{Projct}. 

We have the exact sequences $$0\longrightarrow H(\widetilde{E}( M[z]))[z-1]\longrightarrow H(\widetilde{E}( M[z])) \longrightarrow (z-1)H(\widetilde{E}( M[z]))\longrightarrow 0.$$
 
and
 
 $$0\longrightarrow (z-1)H(\widetilde{E}( M[z]))\longrightarrow H(\widetilde{E}( M[z])) \longrightarrow H(E( M)) \longrightarrow 0.$$
 
 As $H(E( M))$ is $G$-cohomologically trivial, $ (z-1)H(\widetilde{E}( M[z]))) $ and $H(\widetilde{E}( M[z]))$ have the same cohomology. It implies that $H(\widetilde{E}( M[z]))[z-1]$ is $G$-cohomologically trivial. 
Furthermore, we can show that the isomorphism in Proposition \ref{iso} is an isomorphism of $A[G]$-modules. Thus $\dfrac{U(E(M))}{U_{St}(E(M))}$ is $G$-cohomologically trivial. As $U(E(M))$ is also $G$-cohomologically trivial, it is the same for ${U_{St}(E(M))}$. 
As $A$ is a Dedekind ring and that $U_{St}(E(M))$ is a projective $A$-module, it is a projective $A[G]$-module by Theorem \ref{Projct}.

As $U_{St}(E(M))$ is a projective $A[G]$-module, by Theorem \ref{sousmod}, $$[ \Lie_E(M):U_{St}(E(M))]_{A[G]}=\mathcal{L}_G(E( M)).$$ 
\end{proof}

It is the case in particular when $H(E(M))$ is trivial.

\begin{Cor}
If $p\not \vert \hspace{0.1cm}\vert G\vert$ then $U_{St}(\phi(M))$ is $A[G]$-projective and $$[ \Lie_E(M):U_{St}(E(M))]_{A[G]}=\mathcal{L}_G(E( M)).$$
\end{Cor}
In particular, if we take $M=\mathcal{O}_L$, we find the result of Anglès and Taelman in \cite{ref10}  and Anglès and Tavares Ribeiro in \cite{ref3} for Drinfeld modules and the result of Fang in \cite{ref20} for Anderson modules. 
\bigskip

\begin{corollary}\label{triv} 
We denote $N=\Tr_G(M)$.  If $H(E(N))$ is trivial, then $U(E(M))$ and $U_{St}(E(M))$ are projective $A[G]$-modules.  We have

 $$[ \Lie_E(M):U_{St}(E(M))]_{A[G]}=\mathcal{L}_G(E( M)).$$
 
Furthermore, $$[ \Lie_E(M):U(E(M))]_{A[G]}\vert H(E(M))\vert_{A[G]}=\mathcal{L}_G(E( M)).$$
\end{corollary}
\begin{proof}
We have supposed $H(E(N))$ trivial. So we have $\exp_E(K_\infty^n)+N^n=K^n_\infty$. First we show that $H(E(M))$ is $G$-cohomologically trivial. 

We have the exact sequence  $$0\longrightarrow M^n+\exp_E(L_\infty^n)\longrightarrow L^n_\infty\longrightarrow H(E(M))\longrightarrow0.$$ As $L_\infty$ is $G$-cohomologically trivial, to show that $H(E(M))$ is $G$-cohomologically trivial, it suffices to show that $M^n+\exp_\phi(L_\infty^n)$ is. 

We have $$\Tr_G(M^n+\exp_E(L_\infty^n))=N^n+\exp_E(K_\infty^n)=K^n_\infty.$$
Furthermore,  if we look at the $G$-invariants we obtain $$(M^n+\exp_E(L_\infty^n))^G\subset (L^n_\infty)^G=K^n_\infty.$$  The following inclusion $\Tr_G(M^n+\exp_E(L_\infty^n))\subset (M^n+\exp_E(L_\infty^n))^G$ implies  $$\Tr_G(M^n+\exp_E(L_\infty^n))= (M^n+\exp_E(L_\infty^n))^G=K^n_\infty.$$ So we have $\widehat{H}^0(G,M^n+\exp_E(L_\infty))=\{0\}.$ By \cite[Thm. 6 p112]{ref15}, as there is one group of cohomology which is trivial, $M^n+\exp_E(L_\infty^n)$ is $G$-cohomologically trivial. So we have that $H(E(M))$ is $G$-cohomologically trivial. By Proposition \ref{proj},  $$[ \Lie_E(M):U_{St}(E(M))]_{A[G]}=\mathcal{L}_G(E( M)).$$ 

It implies that $$[ \Lie_E(M):U(E(M))]_{A[G]}[ U(E(M)):U_{St}(E(M))]_{A[G]}=\mathcal{L}_G(E( M)).$$ 

By \cite[Theorem 6.2.1]{ref4}, we have $$\dfrac{1}{[ \Lie_E(M):U(E(M))]_{A[G]}}\mathcal{L}_G(E( M))\in \Fitt_{A[G]}(H(E(M)).$$
Thus $[ U(E(M)):U_{St}(E(M))]_{A[G]}=\vert \dfrac{U(E(M))}{U_{St}(E(M))}\vert_{G}\in \Fitt_{A[G]}(H(E(M))$. 
Furthermore, $\vert \dfrac{U(E(M))}{U_{St}(E(M))}\vert_{A}=\vert H(E(M)\vert_A$.
As the same way as before, it gives us $\vert \dfrac{U(E(M))}{U_{St}(E(M))}\vert_{G}=\vert  H(E(M)\vert_{G}$ and so the desired result. 
\end{proof}

\section{An exemple of an Artin-Schreier extension}
In this section, we will see an example of a $L$-function associated to a lattice which is not contained in the ring of integers of an Artin–Schreier extension of a function field.
 For this section, we take $q=p$. 
 
First, we recall the definition of the {\it $d$-th power residue symbol}. 

\begin{definition}Let $b\in A$, $P\in A$ be an irreducible polynomial and $d$ a divisor of $p-1$.\begin{itemize}
\item If $P\not \vert b$, we set $\left(\dfrac{b}{P}\right)_d$ the unique element of $\F_p^*$ such that $$b^{\frac{p^{\deg(P)}-1}{d}}\equiv \left(\dfrac{b}{P}\right)_d \mod P.$$
\item If $P \vert b$, we set $\left(\dfrac{b}{P}\right)_d=0.$ 
\end{itemize}

\end{definition}

It can be extended to all non zero elements of $A$. Let $c=sgn(c)\prod_{i=1}^s P_i^{f_i}$ be the prime decomposition of $c\in A^*$. Then for $b\in A$, the $d$-th power residue symbol is defined as 

$$\left(\dfrac{b}{c}\right)_d=\prod_{i=1}^s \left(\dfrac{b}{P_i}\right)_d^{f_i}.$$
We refer the reader to \cite[Chapter 3]{ref12} for more details.

 Let $C$ be the Carlitz module : the $\F_p$-morphism such that $C_\theta=\theta+\tau$.

We can write $$C_P(x)=\displaystyle\sum_{i=0}^{\deg(P)}[P,i]x^{p^i}$$ where $[P,i]\in A$, $[P,\deg(P)]=1$ and $\forall i\in [\![0; \deg(P)-1 \vert ]\!], P~\vert ~[P,i]$ (see \cite[Chapter 12]{ref12}). 
Thus $C_P(x)\equiv x^{p^{\deg(P)}}\mod P$.

Let $Q(x)=x^p-x-\dfrac{1}{\theta}\in \F_p[\theta]$. We denote by $L$ the decomposition field of $Q$ i.e., $L=\F_p(\alpha)$ where $\alpha$ is a root of $P$. We set $G=\Gal(L/k)$. We can see that it is isomorphic to $\F_p$. 

By \cite[Theorem 2.1]{bae}, we can show that $\mathcal{O}_L=A\oplus_{i=1}^{p-1}\theta A\alpha^i$ and $\theta$ is the only prime which is ramified in $L/k$. Now, we want to prove that $\theta\alpha^{p-1}$ is a normal basis. For $i\in \F_p$, $$\sigma_i(\theta\alpha^{p-1})=\theta(\alpha+i)^{p-1}=\theta\sum_{j=0}^{p-1}\binom{p-1}{j}i^j\alpha^{p-1-j}.$$

If we express $\sigma_i(\theta\alpha^{p-1})$ in function of $1$, $\theta\alpha$, $\theta\alpha^2$, \ldots, $\theta\alpha^{p-1}$, we obtain the following matrix  
$\begin{pmatrix}
0 & \theta \binom{p-1}{p-1}& \theta \binom{p-1}{p-1}2^{p-1} &\ldots & \theta \binom{p-1}{p-1}(p-1)^{p-1}\\
0 &  \binom{p-1}{p-2}& \binom{p-1}{p-2}2^{p-2}& \ldots & \binom{p-1}{p-2}(p-1)^{p-2}\\
0 &  \binom{p-1}{p-3}& \binom{p-1}{p-3}2^{p-3}& \ldots & \binom{p-1}{p-3}(p-1)^{p-3}\\
 \vdots & \vdots & \vdots & & \vdots \\
0 &  \binom{p-1}{1}& \binom{p-1}{1}2& \ldots & \binom{p-1}{1}(p-1)\\
1 &  1& 1& \ldots & 1\\
\end{pmatrix}$.
As its determinant is invertible in $k$, we have $L=\theta k[G]\alpha^{p-1}$. 
  We set $M=\theta A[G]\alpha^{p-1}$. From this matrix, we obtain $M=\oplus_{i=0}^{p-1}\theta A\alpha^i$.
We denote $\delta=\sqrt[p-1]{1+\theta^{p-1}}$ and $M'=M\delta=\theta A[G]\alpha^{p-1}\delta$.

 By choice of $\delta$, we can remark that $M'\not\subset \mathcal{O}_L$ and $M'$ is an almost taming module for $L/k$.   
 
 First, we will look at the $\mathcal{L}$-function attached to $M'$. 

We have $(\delta\theta \alpha^{p-1})^p=\delta^p\theta^p(\alpha+\dfrac{1}{\theta})^{p-1}=(1+\theta^{p-1})\delta\theta(\theta\alpha+1)^{p-1}. $ So $\tau(\theta\delta \alpha^{p-1})\equiv \theta\delta \mod \theta M'.$ By \cite[Lemma 2.3]{bae} $\Tr_G(\theta\alpha^{p-1})=-\theta$. Therefore, $C(M'/\theta M')$ is annihilated by $\theta +\Tr_G$.

Let $P\in \Spec(A)\setminus\theta$ and $\sigma_P$ be a Frobenius associated to $P$.

We recall that $C_P\equiv \tau^{\deg(P)} \mod P$.   Then we have
\begin{align*} 
C_{P}(\theta \delta \alpha^{p-1}) &\equiv \theta^{p^{\deg(P)}}\delta^{p^{\deg(P)}}(\alpha^{p-1})^{p^{\deg(P)}}\mod P\\
&\equiv \theta (1+\theta^{p-1})^{\frac{p^{\deg(P)}}{p-1}}\sigma_P(\alpha^{p-1}) \mod P\\
&\equiv \theta \delta\left(\dfrac{1+\theta^{p-1}}{P}\right)_{p-1}\sigma_P(\alpha^{p-1}) \mod P. \notag
\end{align*}
    Thus we obtain $\vert C(M'/P M')\vert_G= P1_G-\left(\dfrac{1+\theta^{p-1}}{P}\right)_{p-1}\sigma_P$. 
It implies that we get the  following  equivariant  $\mathcal{L}$-function  : 
\begin{align*} 
\mathcal{L}_G(C( M')) &= \dfrac{\vert \Lie_{C}(M'/\theta M')\vert_G }{\vert C(M'/\theta M')\vert_G}\prod\limits_{P\in \MSpec(A)\setminus \{\theta\}}\dfrac{\vert \Lie_{C}(M'/PM')\vert_G }{\vert C(M'/PM')\vert_G} \\
&=\dfrac{\theta}{\theta+\Tr_G}\prod\limits_{\substack{
P\in \MSpec(A)\setminus \{\theta\} }}\dfrac{P}{P-\left(\tfrac{1+\theta^{p-1}}{P}\right)_{p-1}\sigma_P}\notag \\
&=(1-\dfrac{\Tr_G}{\theta})\sum\limits_{\substack{
a\in A^+,(a,\theta)=1 \\
(a,1+\theta^{p-1})=1}}\dfrac{\left(\tfrac{1+\theta^{p-1}}{a}\right)_{p-1}\sigma_a}{a}\\
&=\sum\limits_{\substack{
a\in A^+,(a,\theta)=1 \\
(a,1+\theta^{p-1})=1}}\frac{\left(\tfrac{1+\theta^{p-1}}{a}\right)_{p-1}\sigma_a}{a}-\left(\sum\limits_{\substack{
a\in A^+,(a,\theta)=1 \\
(a,1+\theta^{p-1})=1}}\frac{\left(\tfrac{1+\theta^{p-1}}{a}\right)_{p-1}}{a\theta}\right)\Tr_G. \notag
\end{align*}

As $k_\infty=A\oplus\theta^{-1}\F_p[\![\theta^{-1}]\!]$,  we have $L_\infty=M'\oplus \oplus_{i=0}^{p-1} \F_p[\![\dfrac{1}{\theta}]\!]\delta\alpha^i$ and $\oplus_{i=0}^{p-1} \F_p[\![\dfrac{1}{\theta}]\!]\delta\alpha^i\subset \exp_C(L_\infty)$ thus $H(C(M'))=\{0\}$.
It follows that $U_{St}(C(M'))=U(C(M))$ and it is $A[G]$-projective. Then 

$$[M:U(C(M))]_{A[G]}=\sum\limits_{\substack{
a\in A^+,(a,\theta)=1 \\
(a,1+\theta^{p-1})=1}}\dfrac{\left(\tfrac{1+\theta^{p-1}}{a}\right)_{p-1}\sigma_a}{a}-\left(\sum\limits_{\substack{
a\in A^+,(a,\theta)=1 \\
(a,1+\theta^{p-1})=1}}\dfrac{\left(\tfrac{1+\theta^{p-1}}{a}\right)_{p-1}}{a\theta}\right)\Tr_G.$$ 

Now, we look at the $\mathcal{L}$-function attached to $\widetilde{M}$. 

We have the isomorphism of  $\widetilde{A}[G]$-modules $\dfrac{\widetilde{M'}}{P\widetilde{M'}}\cong \dfrac{M'}{PM'}[z]\otimes_{\F_p[z]}\F_p(z)$.
As $\dfrac{M'}{PM'}[z]$ is a free
 $\dfrac{A}{PA}[z][G]$-module, we get
  $\vert\dfrac{\widetilde{M'}}{P\widetilde{M}}\vert_G=P1_G$.
  
By the same arguments used previously, we have  $\widetilde{C}_P\equiv z^{\deg(P)}\tau^{\deg(P)} \mod P$.
As the same way, we can show that for $P\in \Spec(A)\setminus \theta$, $\vert\widetilde{C}(\widetilde{M'}/P\widetilde{M'})\vert_G=P1_G-z^{\deg(P)}\left(\dfrac{1+\theta^{p-1}}{P}\right)_{p-1}\sigma_P$ and  $\vert\widetilde{C}(\widetilde{M'}/\theta\widetilde{M'})\vert_G=\theta1_G+z\Tr_G$.

So we obtain the $\mathcal{L}$-function :$$ \tiny{\mathcal{L}_G(\widetilde{C}(\widetilde{M'}))=\sum\limits_{\substack{
a\in A^+,(a,\theta)=1 \\
(a,1+\theta^{p-1})=1}}\tfrac{z^{\deg(a)}\left(\tfrac{1+\theta^{p-1}}{a}\right)_{p-1}\sigma_a}{a}-\left(\sum\limits_{\substack{
a\in A^+,(a,\theta)=1 \\
(a,1+\theta^{p-1})=1}}\tfrac{z^{\deg(a)+1}\left(\tfrac{1+\theta^{p-1}}{a}\right)_{p-1}}{a\theta}\right)\Tr_G.}$$

To finish this exemple, we look at the $\mathcal{L}$-functions associated to $N$ where  $N=\Tr_G(M')$. As  $\Tr_G(\theta\alpha^{p-1})=-\theta$  by \cite[Lemma 2.3]{bae}, $N=\theta A \delta$. 
For a prime that does not divide $1+\theta^{p-1}$, $C(N/PN)$ is annihilated by $P-\left(\dfrac{1+\theta^{p-1}}{P}\right)_{p-1}$. Indeed, for $x\in N$, we can write $x=\theta a \delta$ with $a\in A$. Thus 

\begin{align*} 
C_{P-\left(\tfrac{1+\theta^{p-1}}{P}\right)_{p-1}}(x) &\equiv x^{p^{\deg(P)}}-\left(\dfrac{1+\theta^{p-1}}{P}\right)_{p-1}a\theta \delta \mod P\\
&\equiv a\theta \delta^{p^{\deg(P)}}-\theta a\sqrt[p-1]{1+\theta^{p-1}}\sqrt[p-1]{1+\theta^{p-1}}^{p^{\deg(P)}-1} \mod P\\
&\equiv 0 \mod P. \notag
\end{align*}

As the same way, we can show that if $P$ divides $1+\theta^{p-1}$, then $C(N/PN)$ is annihilated by $P$.
We obtain the $\mathcal{L}$-function : 
\begin{align*} 
\mathcal{L}(C( N)) &= \prod\limits_{P\in \MSpec(A)}\dfrac{\vert \Lie_{C}(N/PN)\vert_A }{\vert C(N/PN)\vert_A} \\
&=\prod\limits_{\substack{
P\in \MSpec(A) \\
(P,1+\theta^{p-1})=1}}\dfrac{P}{P-\left(\frac{1+\theta^{p-1}}{P}\right)_{p-1}}\notag \\
&=\sum\limits_{\substack{
b\in A^+ \\
(b,1+\theta^{p-1})=1}}\dfrac{\left(\frac{1+\theta^{p-1}}{b}\right)_{p-1}}{b}. \notag
\end{align*}

As $k_\infty=A\oplus\theta^{-1}\F_p[\![\theta^{-1}]\!]=N\oplus  \F_p[\![\theta^{-1}]\!]b$ and $\F_p[\![\theta^{-1}]\!]\delta\subset \exp_C(k_\infty)$, we obtain $H(C(N))=\{0\}$. It means that $U(C(N))=U_{St}(N)$ and 

$$U_{St}(C(\theta\F_p[\theta]\sqrt[p-1]{1+\theta^{p-1}}))=\sum\limits_{\substack{
b\in A^+ \\
(b,1+\theta^{p-1})=1}}\dfrac{\left(\frac{1+\theta^{p-1}}{b}\right)_{p-1}}{b} \theta\F_p[\theta]\sqrt[p-1]{1+\theta^{p-1}}.$$

As the same way, it means that for $P$ which does not divide $1+\theta^{p-1}$, $\widetilde{C}(\widetilde{N}/P\widetilde{N})$ is annihilated by $P-\left(\dfrac{1+\theta^{p-1}}{P}\right)_{p-1}z^{\deg(P)}$. It implies that $$\mathcal{L}(\widetilde{C}( \widetilde{N}))=\sum\limits_{\substack{
b\in A^+ \\
(b,1+\theta^{p-1})=1}}\dfrac{\left(\frac{1+\theta^{p-1}}{b}\right)_{p-1}z^{\deg(b)}}{b} $$

and $$U(\widetilde{C}(\theta\F_q(z)[\theta]\delta) =\sum\limits_{\substack{
b\in A^+ \\
(b,1+\theta^{p-1})=1}}\dfrac{\left(\frac{1+\theta^{p-1}}{b}\right)_{p-1}z^{\deg(b)}}{b} \theta\F_q(z) [\theta]\delta.$$

%%%%%%%%%%%%%%%%%

%%%%%%%%%%%%%%%%%%%

\end{document}